\documentclass[10pt,reqno]{amsart}
\usepackage[margin=1.0in]{geometry}
\usepackage{amssymb,mathrsfs}
\usepackage{ifthen}
\usepackage{colortbl}
\usepackage{cases,appendix}
\definecolor{black}{rgb}{0.0, 0.0, 0.0}

\provideboolean{shownotes} 
\setboolean{shownotes}{true} 
\newcommand{\margnote}[1]{
\ifthenelse{\boolean{shownotes}}%
{\marginpar{\raggedright\tiny\texttt{#1}}}%
{}%
}
\newcommand{\hole}[1]{
\ifthenelse{\boolean{shownotes}}%
{\begin{center} \fbox{ \rule {.25cm}{0cm} \rule[-.1cm]{0cm}{.4cm}
\parbox{.85\textwidth}{\begin{center} \texttt{#1}\end{center}} \rule
{.25cm}{0cm}}\end{center}} {} }

\title[Strong solutions to the inhomogeneous Navier-Stokes-BGK system]{Strong solutions to the inhomogeneous Navier-Stokes-BGK system}

\author[Choi]{Young-Pil Choi}
\address[Young-Pil Choi]{\newline Department of Mathematics\newline
Yonsei University, Seoul 03722, Republic of Korea}
\email{ypchoi@yonsei.ac.kr}

\author[Lee]{Jaeseung Lee}
\address[Jaeseung Lee]{\newline The Research Institute of Basic Sciences \newline Seoul National University, Seoul 08826, Republic of Korea}
\email{jaeseunglee@snu.ac.kr}

\author[Yun]{Seok-Bae Yun}
\address[Seok-Bae Yun]{\newline Department of Mathematics \newline
Sungkyunkwan University, Suwon 440-746, Republic of Korea}
\email{sybun01@skku.edu}

\subjclass{}

\keywords{}

\numberwithin{equation}{section}

\newtheorem{theorem}{Theorem}[section]
\newtheorem{lemma}{Lemma}[section]

\newtheorem{proposition}{Proposition}[section]
\newtheorem{remark}{Remark}[section]
\newtheorem{definition}{Definition}[section]

\newcommand{\R}{\mathbb R}

\newcommand{\ls}{\lesssim}

\newcommand{\T}{\mathbb T}

\newcommand{\bq}{\begin{equation}}
\newcommand{\eq}{\end{equation}}
\newcommand{\e}{\varepsilon}
\newcommand{\lt}{\left}
\newcommand{\rt}{\right}

\newcommand{\pa}{\partial}

\begin{document}
\allowdisplaybreaks

\date{\today}


\begin{abstract}
In this paper, we are concerned with the local-in-time well-posedness of a fluid-kinetic model in which the BGK model with density dependent collision frequency is coupled with the inhomogeneous Navier-Stokes equation through drag forces. To the best knowledge of authors, this is the first result on the existence of local-in-time smooth solution for particle-fluid model with nonlinear inter-particle operator for which the existence of time can be prolonged as the size of initial data gets smaller.
\end{abstract}

\keywords{Fluid-Kinetic model, BGK model, inhomogeneous Navier-Stokes equations, spray models,
global existence, strong solutions.}


\maketitle \centerline{\date}

\tableofcontents

%
%
%
%
\section{Introduction} \label{sec:1}
Sprays are complex flows consisting of dispersed particles in underlying gas, for instances, spray in the air, fuel-droplets suspended in the cylinder in the combustion process of engines, pollutants floating in the air or water. The evolution of such particle-fluid system can be described in various ways according to the corresponding physical situation and the modeling assumptions. In this paper, we consider the case where the relaxation through inter-particle collisions and the drag of the surrounding fluid compete, which is described by the BGK model coupled with the inhomogeneous Navier-Stokes equations through drag forces:
\begin{gather}\begin{gathered} \label{main}
\pa_t f + v \cdot \nabla_x f + \nabla_v \cdot((u-v) f) = \rho_f(\mathcal M(f)-f), \\
\pa_t \rho + \nabla_x \cdot (\rho u) = 0,\\
\pa_t (\rho u) + \nabla_x \cdot (\rho u \otimes u) + \nabla_x p - \mu \Delta_x u = -\int_{\R^3}(u-v)f \,dv, \\
\nabla_x \cdot u = 0,
\end{gathered}\end{gather}
subject to initial data:
\begin{align}\begin{aligned}\label{ini_main}
(f(x,v,0),\rho(x,0), u(x,0)) =:(f_0(x,v),\rho_0(x), u_0(x)), \quad (x,v) \in \T^3 \times \R^3.
\end{aligned}\end{align}
Here, $f = f(x,v,t)$ denotes the number density function of the immersed particles on the phase space of position $x \in \T^3$ and velocity $v \in \R^3$ at time $t > 0$ , and $\rho = \rho(x,t)$ and $u = u(x,t)$ are the local density and bulk velocity of the fluid, respectively. For simplicity, we assume that the viscosity coefficient $\mu = 1$ throughout the paper.  The local Maxwellian $\mathcal M(f)$ is defined by
\[
\mathcal M(f)(x,v,t) = \frac{\rho_f(x,t)}{\sqrt{(2\pi T_f(x,t))^3}} \exp \left(-\frac{|v-U_f(x,t)|^2}{2T_f(x,t)} \right),
\]
where the macroscopic fields of local particle density $\rho_f$, local particle velocity $U_f$, and local particle temperature $T_f$ are given by
\begin{align*}
\rho_f(x,t) &:= \int_{\R^3}f(x,v,t)\,dv,\\
\rho_f(x,t)U_f(x,t) &:= \int_{\R^3}vf(x,v,t)\,dv, \quad \mbox{and} \\
3\rho_f(x,t)T_f(x,t) &:= \int_{\R^3}|v-U_f(x,t)|^2 f(x,v,t)\,dv.
\end{align*}
An explicit computation gives the following cancellation property:
\[
\int_{\R^3}(\mathcal M(f) - f)\begin{pmatrix}
1\\v\\|v|^2
\end{pmatrix}dv = 0,
\]

Particle-fluid models have received immense attention recently since the situation of particles drafting in fluid arises very often in nature or engineering, and the coupling of kinetic equations and fluid equations addresses various interesting mathematical problems and modeling issues. We can roughly divide the literature on the mathematical theory of such kinetic-fluid model into two categories according to whether
the collisional interactions between the immersed particles are taken into account or not. In the absence of collisional interactions,
Vlasov or Vlasov-Fokker-Planck type equations coupled with various fluid equations are investigated. For the existence of
the weak solutions of such collisionless particle-fluid models, we refer to \cite{BDGM, CCK, CK, Ha, MV, WY}. Results on the strong solutions can be found in \cite{CDM, CKL}. Particle-kinetic models involving local-alignment phenomena between the immersed particles can be found in \cite{BCHK,BCHK14, CL}.
We now turn to literature including particle-particle collisions.
In \cite{BDM,YY} the existence of weak solutions for Vlasov-Navier-Stokes equations with a linear particle operator that explains the break-up of droplets is considered. In \cite{M}, Mathiaud obtained the existence of local-in-time classical solution for the Navier-Stokes-Boltzmann equation when the initial data is a small perturbation of a global Maxwellian. In \cite{CY}, the authors obtained the existence of global-in-time existence of weak solutions under the condition of finite mass, energy and entropy. In \cite{Choi16, Choi17}, large-time behavior of solutions and finite-time blow-up phenomena of particle-fluid systems are considered.

A brief review on the BGK model is also in order. The BGK models \cite{BGK} have been very popularly employed in physics and engineering as a satisfactory relaxational approximation of the Boltzmann equation which suffers severely from high computational cost. The existence theory for the BGK model is first established by Perthame \cite{Pe} in which the weak solution is obtained under the condition of finite mass momentum and energy. For the initial data with appropriate decay in the velocity space, a unique existence is established in \cite{PP93}. These results are adapted and extended, for example, to $L^p$ problem \cite{ZH}, gases under the influence of external forces or mean-fields \cite{Zhang}, gas mixture problem in which the gas consists of more one type of gas molecules \cite{KP}, ellipsoidally generalized BGK model introduced to better calibrate fluid coefficients \cite{Yun2}, and polyatomic molecules formed by bonds of more than one atom \cite{PY,Yun21}. The existence of classical solution near equilibrium and their asymptotic equilibrization can be found in \cite{Yun1, Yun3}. For the studies on the stationary problems for the BGK model, see \cite{BY,Ukai-BGK}. BGK model is also fruitfully employed in the derivation of various macroscopic or hydrodynamic models \cite{V,DMOS,LT,Mellet,MMM,SR,SR1,SR2}.
The literature on the numerical applications of the BGK model are immense, we refer to \cite{CP,DP,FJ,M,PP07,RSY,XH} and references therein for interested readers.

To the best knowledge of the authors, the only result on the existence of classical solutions for particle-kinetic models involving collisional interactions between immersed particle is established in \cite{M} (for weak solutions, see \cite{CY}), in which
Mathiaud considers a local-in-time existence for a fluid-kinetic model constructed from the coupling of the Navier-Stokes equation with the Boltzmann equation near a
global Maxwellian under the assumption that the high order energy functional is sufficiently small. In \cite{M}, however, the exchange between the length of the life span and the size of the initial data does not occur. That is, no matter how small an initial perturbation we take in the energy norm,  the life span of the solution cannot be extended over a certain fixed time. In this paper, we show that such restriction can be removed, at least for the case of the BGK type relaxation operator. We also mention that the global-in-time existence of strong solution for the relaxation operator with nontrivial collision frequency remains open even for the non-coupled classical BGK model.

To precisely state our main result, we first define the notion of a strong solution.
\begin{definition} \label{def}
For a given time $T \in (0,\infty)$, we say that $(f,\rho,u)$ is a strong solution to system \eqref{main}-\eqref{ini_main} if it satisfies the system in the sense of distributions with the following regularity:
\begin{eqnarray} \label{reg}
&&(i)~ f \in \mathcal C([0,T];W_q^{1,\infty}(\T^3 \times \R^3)) \mbox{ with }q>5, \cr
&&(ii)~ \rho \in \mathcal C([0,T];H^3(\T^3)), \cr
&&(iii)~ u \in \mathcal C([0,T];H^2(\T^3)) \cap L^2(0,T;H^3(\T^3)).
\end{eqnarray}
\end{definition}

Our main results read as follows (see Notation below the statement of the theorem for the definitions of function spaces):

\begin{theorem} \label{thm_main}
Fix $T \in (0,\infty)$. Then, there exists $\varepsilon>0$, which depends only on $T$, such that for any initial data $(f_0,\rho_0,u_0)$ satisfying the following conditions:
\begin{eqnarray*}
&& (i)~\inf_{x \in \T^3} \rho_0(x) > 0, \quad \rho_0  \in H^3(\T^3), \cr
&& (ii)~\sum_{|\nu|\leq 1} ess\sup_{x,v}(1+|v|)^q |\nabla^\nu f_0(x,v)| + \|u_0\|_{H^2(\T^3)} < \varepsilon, \quad \mbox{and}\cr
&& (iii)~f_0 > \varepsilon_1(1+|v|)^{-(q+3+a)}, \quad \mbox{for some} \quad\e_1 >0 \mbox{ and }a>0,
\end{eqnarray*}
the system \eqref{main}-\eqref{ini_main} admits the unique strong solution $(f, \rho, u)$.
\end{theorem}

\begin{remark}
The initial positivity condition (iii) is necessary to guarantee the positivity of macroscopic field $\rho_{f}$, see Lemma \ref{lem_macro_bdd}.
\end{remark}

\noindent {\bf Notation.} Throughout the paper, $\nabla^k$ denotes any partial derivative $\pa^\alpha$ with multi-index $\alpha$, $|\alpha|=k$. We often omit $x$-dependence of differential operators for simplicity of notation. We denote by C a generic, not necessarily identical, positive constant. The relation $A \lesssim B$ denotes the inequality $A \leq CB$ for such a generic constant. Below we introduce the norms and function spaces to be used in the paper.
\begin{quote}
$\bullet$ For functions $f(x, v), g(x)$, $\|f\|_{L^p}$ and $ \|g\|_{L^p}$ denote the usual $L^p(\T^3 \times \R^3)$-norm and $L^p(\T^3)$-norm, respectively.\\
$\bullet$ We use the following weighted norms for $f(x,v)$:

\begin{align*}
&\|f\|_{q} := \|f\|_{L^\infty_q} := ess\sup_{x,v}(1+|v|)^q f(x,v) ,\quad \|f\|_{W_q^{1,\infty}} := \sum_{|\nu| \leq 1}\| \nabla^\nu f\|_q.
\end{align*}
 $L_q^\infty(\T^3 \times \R^3)$ and $ W_q^{1,\infty}(\T^3 \times \R^3)$ naturally denote the spaces of functions with finite corresponding norms.
\\
$\bullet$  $H^s(\T^3)$ denotes the $s$-th order $L^2(\T^3)$ Sobolev space.
\end{quote}

The rest of the paper is organized as follows. In Section \ref{sec:2}, we introduce several lemmas regarding boundedness properties of the macroscopic fields $(\rho_f, U_f, T_f)$ and the local Maxwellian $\mathcal M(f)$, which will be heavily used throughout the paper. In Section \ref{sec:3}, a sequence of approximation systems to \eqref{main}-\eqref{ini_main} is constructed. In Section \ref{sec:4}, we prove that the sequence of solutions constructed in Section \ref{sec:3} is indeed a Cauchy sequence and the limit is the solution of the system \eqref{main} in the sense of Definition \ref{def}. 

%
%
%
%
\section{Preliminaries} \label{sec:2}
We present a series of lemmas that will be crucially used throughout the paper.

\begin{lemma} \cite{PP93} \label{lem_temp}
There exists a positive constant $C_q$, which depends only on $q$, satisfying
\begin{itemize}
\item[(i)] $\rho_f \leq C_q \|f\|_q T_f^{3/2}  \quad (q>3~\mbox{or}~q=0)$,
\item[(ii)] $\rho_f(T_f + |U_f|^2)^{(q-3)/2} \leq C_q \|f\|_q \quad (q>5~\mbox{or}~q=0)$,
\item[(iii)] $\rho_f|U_f|^{q+3}((T_f+|U_f|^2)T_f)^{-3/2} \leq C_q \|f\|_q. \quad (q>1~\mbox{or}~q=0)$,
\end{itemize}
for almost everywhere $x \in \T^3$.
\end{lemma}
We now show that the $\|\cdot\|_q$-norm of a generalized local Maxwellian $\mathcal M_\gamma(f)$ with $\gamma > 0$ can be controlled by that of $f$. Although the proof is essentially given in \cite{PP93}, we provide it here for the completeness of our present work.
\begin{lemma} \label{lem_M_bdd}
Suppose $\|f\|_q < \infty$ for $q>5$, and let $\gamma >0$ be given. Then there exists a positive constant $C_{q,\gamma}$, which depends only on $q$ and $\gamma$, such that
\[
\|\mathcal M_\gamma(f)\|_q \leq C_{q,\gamma} \|f\|_q, \quad (q>5~\mbox{or}~q=0),
\]
where
\[
\mathcal M_\gamma(f) :=  \frac{\rho_f}{\sqrt{(2\pi T_f)^3}}\exp\left(-\gamma\frac{|v-U_f|^2}{2T_f}\right).
\]
In particular, if $\gamma=1$, then $\mathcal{M}_1(f) = \mathcal{M}(f)$ and
\[
\| \mathcal M(f)\|_q \leq C_q \|f\|_q, \quad (q>5~\mbox{or}~q=0).
\]
\end{lemma}
\begin{proof}
We provide the estimates on $\mathcal M_\gamma(f)$ and $|v|^q \mathcal M_\gamma(f)$, seperately. \\
\noindent $\bullet$ (Estimate of $\mathcal M_\gamma(f)$): It follows from Lemma \ref{lem_temp} (i) that
\begin{align*}
\mathcal M_\gamma(f) \leq \frac{\rho_f}{\sqrt{(2\pi T_f)^3}} \leq C_q \|f\|_q.
\end{align*}
\noindent $\bullet$ (Estimate of $|v|^q \mathcal M_\gamma(f)$): We first estimate
\begin{align*}
|v|^q \mathcal M_\gamma(f) \leq C_q \left(|U_f|^q + |v-U_f|^q \right) \mathcal M_\gamma(f) =: \mathcal I_1 + \mathcal I_2,
\end{align*}
where $\mathcal I_1$ can be bounded as
\begin{align*}
\mathcal I_1 = C_q |U_f|^q  \frac{\rho_f}{\sqrt{(2\pi T_f)^3}}\exp\left(-\gamma\frac{|v-U_f|^2}{2T_f}\right) \leq C_q \frac{|U_f|^q \rho_f}{T_f^{3/2}}.
\end{align*}
We now estimate $\mathcal I_1$ by considering two cases: $|U_f| > T_f^{1/2}$ and $|U_f| \leq T_f^{1/2}$.
If $|U_f| > T_f^{1/2}$, we have
\begin{align*}
\mathcal I_1 \leq C_q \frac{|U_f|^{q+3} \rho_f}{|U_f|^3 T_f^{3/2}} \leq C_q \frac{|U_f|^{q+3} \rho_f}{(T_f+|U_f|^2)^{3/2}T_f^{3/2}} \leq C_q \|f\|_q,
\end{align*}
where we used Lemma \ref{lem_temp} (iii) for the last inequality. On the other hand, if $|U_f| \leq T_f^{1/2}$, we use Lemma \ref{lem_temp} (ii) to get
\begin{align*}
\mathcal I_1 \leq C_q \frac{|U_f|^q \rho_f}{T_f^{3/2}} \leq C_q \rho_f T_f^{\frac{q-3}{2}} \leq C_q \rho_f (T_f + |U_f|^2)^{\frac{q-3}{2}} \leq C_q\|f\|_q,
\end{align*}
due to $q > 5$. For $\mathcal I_2$, we get
\begin{align*}
\mathcal I_2 &= C_q |v-U_f|^q \frac{\rho_f}{\sqrt{(2\pi T_f)^3}}\exp\left(-\gamma\frac{|v-U_f|^2}{2T_f}\right) \\
&= C_q \rho_f T_f^{(q-3)/2} \lt( \lt(\frac{|v-U_f|^2}{2T_f}\rt)^{q/2}\exp\lt(-\gamma\frac{|v-U_f|^2}{2T_f}\rt)\rt) \\
&\leq C_{q,\gamma} \rho_f T_f^{(q-3)/2} \leq C_{q,\gamma} \rho_f (T_f + |U_f|^2)^{(q-3)/2} \leq C_{q,\gamma} \|f\|_q.
\end{align*}
Here, we employed the fact $x^{q/2}e^{-\gamma x} \ls 1$ for all $x \geq 0$ and Lemma \ref{lem_temp} (ii).
Finally, the estimates above yield that
\[
\|\mathcal M_\gamma(f)\|_q \leq ess\sup_{x,v}\lt((1+|v|)^q\mathcal M_\gamma(f)(x,v)\rt) \leq C_{q,\gamma} \|f\|_q.
\]
\end{proof}
\begin{lemma} \cite{Yun2} \label{lem_lip}
Assume $f,g$ satisfy ($h$ denotes either $f$ or $g$)
\begin{itemize}
\item[(i)] $\|h\|_q < C_1$,
\item[(ii)] $\rho_h + |U_h| + |T_h| < C_2$,
\item[(iii)] $\rho_h, T_h > C_3$,
\end{itemize}
for some constants $C_i > 0, i=1,2,3$. Then, we have
\[
\| \mathcal M(f) - \mathcal M(g)\|_q \leq C\|f-g\|_q,
\]
where $C>0$ depends only on $C_i (i=1,2,3)$.
\end{lemma}
\begin{lemma}  \label{lem_macro}
Suppose $\|f\|_{q} < \infty$ for $q>5$, and $\rho_f,U_f$, and $T_f$ satisfy
\[
\rho_f + |U_f| + T_f <c_1 \quad \mbox{and} \quad \rho_f,T_f > c_2,
\]
for some positive constants $c_1$ and $c_2$. Then we have
\[
\|\nabla_{x,v}\mathcal M(f)\|_q \leq C(\| \nabla_x f\|_q +1)\|f\|_q,
\]
where $C$ is a positive constant depending on $c_1$ and $c_2$.
\end{lemma}
\begin{proof}
\noindent We first provide derivatives of the local Maxwellian $\mathcal M(f)$ with respect to the macroscopic fields:
\[
 \frac{\pa \mathcal M (f)}{\pa \rho_f} = \frac{1}{\rho_f}\mathcal M(f), \quad \frac{\pa \mathcal M(f)}{\pa U_f} = \frac{v-U_f}{T_f}\mathcal M(f), \quad \mbox{and} \quad \frac{\pa \mathcal M(f)}{\pa T_f} = \left(-\frac{3}{2T_f} + \frac{|v-U_f|^2}{2T_f^2} \right)\mathcal M(f).
\]
We then give the estimates for $\| \cdot \|_q$-norm of each term above. We easily find
\bq \label{A-3}
\Big\| \frac{\pa \mathcal M(f)}{\pa \rho_f} \Big\|_q \lesssim \| \mathcal M(f)\|_q \lesssim \|f\|_q.
\eq
For the second one, note that
\begin{align*}
\begin{aligned} 
\lt| \frac{v-U_f}{T_f} \rt|\mathcal M(f) &= \sqrt{2}(2\pi)^{-\frac 32}\frac{\rho_f}{T_f^2} \lt| \frac{v-U_f}{\sqrt{2T_f}}   \rt| \exp\left(-\frac{|v-U_f|^2}{2T_f}    \right) \lesssim \frac{\rho_f}{T_f^{3/2}} \exp\left(-\frac{|v-U_f|^2}{4T_f}    \right).
\end{aligned}
\end{align*}
Here, we used the following simple inequality
\begin{equation} \label{simple-1}
xe^{-x^2} \lesssim  e^{-x^2/2} \quad \mbox{for all} \quad x \geq 0.
\end{equation}
Then, we use Lemma \ref{lem_M_bdd} to find
\begin{equation} \label{A-2}
\lt\| \frac{v-U_f}{T_f}\mathcal M(f)\rt\|_q \lesssim \left\|\frac{\rho_f}{T_f^{3/2}}\exp\left(-\frac{|v-U_f|^2}{4T_f}    \right)  \right\|_q     \lesssim \|f\|_q.
\end{equation}
In order to estimate the third one, we use the following inequality similar to \eqref{simple-1}:
\begin{equation*}
xe^{-x} \lesssim e^{-x/2} \quad \mbox{for all} \quad x \geq 0.
\end{equation*}
This yields
\begin{align*}
\frac{|v-U_f|^2}{2T_f^2} \mathcal M(f) \lesssim \frac{\rho_f}{T_f^{3/2}} \exp\lt(-\frac{|v-U_f|^2}{4T_f}  \rt),
\end{align*}
and subsequently, this with Lemma \ref{lem_M_bdd} gives
\[
\lt\| \frac{|v-U_f|^2}{2T_f^2}\mathcal M(f) \rt\|_q  \lesssim \|f\|_q.
\]
Thus we have
\begin{align} \label{A-4}
\lt\| \frac{\pa \mathcal M(f)}{\pa T_f} \rt\|_q \leq \lt\| \frac{3}{2T_f} \mathcal M(f) \rt\|_q + \lt\| \frac{|v-U_f|^2}{2T_f^2}\mathcal M(f) \rt\|_q \lesssim \|f\|_q.
 \end{align}
The first order derivatives of the macroscopic fields are given by
\begin{align*}
\nabla_x \rho_f &= \int_{\R^3} \nabla_x f \,dv, \\
\nabla_x U_f &= \nabla_x \left( \frac{1}{\rho_f} \int_{\R^3}vf \,dv \right) = -\frac{U_f}{\rho_f}\nabla_x \rho_f + \frac{1}{\rho_f} \int_{\R^3}v \nabla_x f \,dv, \\
\nabla_x T_f &= \frac{1}{3} \nabla_x \left(\frac{1}{\rho_f} \int_{\R^3} |v-U_f|^2 f\,dv \right) \\
&= \frac{1}{3}\lt(-\frac{\nabla_x \rho_f}{\rho_f^2} \int_{\R^3}|v-U|^2 f \,dv - \frac{1}{\rho_f}\int_{\R^3}2(v-U_f)f \nabla_x U_f \,dv + \frac{1}{\rho_f}\int_{\R^3}|v-U_f|^2 \nabla_x f \,dv \rt).
\end{align*}
Then we easily get
\begin{align} \label{A-4-1}
|\nabla_x \rho_f | \leq \int_{\R^3} | \nabla_x f|\,dv = \int_{\R^3} |\nabla_x f|(1+|v|)^q (1+|v|)^{-q}\,dv \lesssim \| \nabla_x f\|_q,
\end{align}
due to $q > 5$. Similarly, we also find
\begin{equation*}
| \nabla_x U_f|   \lesssim \| \nabla_x f \|_q \quad \mbox{and} \quad |\nabla_x T_f| \lesssim \| \nabla_x f\|_q.
\end{equation*}
This together with \eqref{A-2}, \eqref{A-3}, and \eqref{A-4} gives
\begin{align*}
\|\nabla_{x,v}\mathcal M(f)\|_q  &\leq \lt\| \frac{\pa \mathcal M(f)}{\pa \rho_f}\rt\|_q \|\nabla_x \rho_f\|_{L^\infty} + \lt\|\frac{\pa \mathcal M(f)}{\pa U_f}\rt\|_q \| \nabla_x U_f\|_{L^\infty} \cr
&\quad + \lt\|\frac{\pa \mathcal M(f)}{\pa T_f}\rt\|_q\|\nabla_x T_f\|_{L^\infty} + \lt\|\frac{|v-U_f|}{T_f}\mathcal M(f) \rt\|_q \\
&\lesssim (\| \nabla_x f\|_q +1) \|f\|_q.
\end{align*}
\end{proof}
%
%
%
%
\section{Global existence and uniqueness of approximation system} \label{sec:3}
We construct the sequence of approximation solutions to linearized systems of \eqref{main}. We consider following linearized NS-BGK system:
\begin{gather}\begin{gathered} \label{app_main}
\pa_t f^{n+1} + v \cdot \nabla_x f^{n+1} + \nabla_v \cdot ((u^n-v)f^{n+1}) = \rho_{f^n}(\mathcal M(f^n) - f^{n+1}),\\
\pa_t \rho^{n+1} + u^n \cdot \nabla_x \rho^{n+1} = 0, \\
\rho^{n+1} \pa_t u^{n+1} + \rho^{n+1}u^n \cdot \nabla_x u^{n+1} - \Delta_x u ^{n+1} + \nabla_x p^{n+1} = -\int_{\R^3}(u^n-v)f^{n+1}\,dv, \\
\nabla_x \cdot u^{n+1} = 0,
\end{gathered}\end{gather}
with the initial data and the first iteration step:
\begin{align} \label{init_app}
\begin{aligned}
(f^{n+1}(x,v,0),\rho^{n+1}(x,0),u^{n+1}(x,0)) &= (f_0(x,v),\rho_0(x),u_0(x)) \quad \mbox{and} \cr
(f^0(x,v,t),\rho^0(x,t),u^0(x,t)) &= (f_0(x,v),\rho_0(x),u_0(x))
\end{aligned}
\end{align}
for $n \geq 0$ and $(x,v,t) \in \T^3 \times \R^3 \times (0,T)$.

We now consider the backward characteristic  $Z^n(s):=(X^n(s),V^n(s)) := (X^n(s;t,x,v),V^n(s;t,x,v))$, $s,t \in [0,T]$ given by
\begin{align} \label{back}
\begin{aligned}
\frac{d}{ds}X^{n+1}(s) &= V^{n+1}(s),\\
\frac{d}{ds}V^{n+1}(s) &= u^n(X^{n+1}(s),s) - V^{n+1}(s),
\end{aligned}
\end{align}
subject to the terminal data:
\[
Z^{n+1}(t) = (x,v) =:z.
\]
We now provide the existence result for the approximation system \eqref{app_main}-\eqref{init_app}.
\begin{proposition} \label{prop}
Let $T \in (0,\infty)$ be an arbitrary fixed number. Suppose that the initial data $(f_0,\rho_0,u_0)$ satisfy the assumptions of Theorem \ref{thm_main}.
Choose $\varepsilon$ to satisfy $\varepsilon^{1-\beta}C_T<1$, where $C_T$ is given in the end of the proof. Then, if $f^n$ and $u^{n}$ satisfy the following conditions:
\begin{equation}
\begin{aligned} \label{n_condition}
&\sum_{|\nu|\leq 1}\|\nabla^\nu f^{n}\|_q < \varepsilon^\beta \quad \mbox{and}\\
&\|\pa_t u^n\|_{\mathcal C([0,T];L^2)} + \|\pa_t u^n\|_{L^2(0,T;H^1)} + \|u^n\|_{\mathcal C([0,T];H^2)} + \|u^n\|_{L^2(0,T;H^3)} < \varepsilon^\alpha,
\end{aligned}
\end{equation}
then there exists a unique solution $(f^{n+1},\rho^{n+1},u^{n+1})$ such that
\begin{align*}
\begin{aligned}
& \sum_{|\nu|\leq 1}\|\nabla^\nu f^{n+1}\|_q < \varepsilon^\beta, \quad \rho  \in \mathcal C([0,T];H^3(\T^3)), \quad \rho(x,t) \geq \delta > 0, \quad \forall (x,t) \in \T^3 \times [0,T], \quad \mbox{and}\\
& \|\pa_t u^{n+1}\|_{\mathcal C([0,T];L^2)} + \|\pa_t u^{n+1}\|_{L^2(0,T;H^1)} + \|u^{n+1}\|_{\mathcal C([0,T];H^2)} + \|u^{n+1}\|_{L^2(0,T;H^3)} < \varepsilon^\alpha,
\end{aligned}
\end{align*}
where $\delta = \inf_{x \in \T^3}\rho_0(x)$ is a positive constant and $\alpha,\beta$ are constants such that $0 < \alpha < \beta < 1$.
\end{proposition}
We first note that the existence and uniqueness of the momentum equations in \eqref{app_main}, which is linear parabolic system, are well-known thanks to the semigroup theory, see \cite{Ka} for instance. We prove Proposition \ref{prop} through the following lemmas.
The next lemma gives the existence of positive lower bound and the regularity of the fluid density. Since the proof is similar to that of \cite[Lemma 2.2]{CK}, we omit it here.
\begin{lemma} \label{lem_rho}
Suppose that the initial data $\rho_0$ and $u^n$ satisfy the assumptions in Theorem \ref{thm_main} and \eqref{n_condition}.  Then, there exists a unique solution $\rho^{n+1}$ to $\eqref{app_main}$ such that
\begin{itemize}
\item[(i)] $\inf_{\T^3 \times [0,T]}\rho^{n+1} \geq \delta$ for some $\delta >0$,
\item[(ii)] $\sup_{0 \leq t \leq T}\|\rho^{n+1}(\cdot,t) \|_{H^3} \leq C\| \rho_0 \|_{H^3}$,
\end{itemize}
where $C>0$ is independent of $n$.
\end{lemma}
Next, we present the growth estimate in velocity for the characteristic flow \eqref{back}.
\begin{lemma} (Estimate of characteristic flow) \label{lem_flow}
Suppose that $u^n$ satisfy \eqref{n_condition}. Then, there exists a constant $C$ depending on  $T$ such that
\begin{eqnarray*}
|V^{n+1}(s)| \leq C(1+|v|), \quad 0 \leq s \leq T.
\end{eqnarray*}
\end{lemma}
\begin{proof}
We rewrite \eqref{back} as
\begin{align}
& X^{n+1}(s) = x - \int_s^t V^{n+1}(\tau)\,d\tau,\notag\\
& V^{n+1}(s) = ve^{t-s} - \int_s^t u^n(X^{n+1}(\tau),\tau)e^{  \tau-s  }\,d\tau. \label{B-1-2}
\end{align}
Note that $u^n$ can be estimated as $ \| u^n\|_{L^\infty} \leq C\|u^n \|_{H^2} < \varepsilon^\alpha,$
where $C>0$ is independent of $n$. Then we easily find from \eqref{B-1-2} that
\[
|V^{n+1}(s)| \leq C(1+|v|),
\]
where $C$ depends on $T$, but independent of $n$.
\end{proof}
The next lemma asserts that the macroscopic fields of kinetic equation have the uniform boundedness property.

\begin{lemma}(Boundedness of macroscopic fields) \label{lem_macro_bdd}
Suppose that the initial data satisfy the assumptions of Theorem \ref{thm_main} and $f^k (1 \leq k \leq n)$  satisfies
 \begin{equation*}
\begin{aligned}
&\sum_{|\nu|\leq 1}\|\nabla^\nu f^{k}\|_q < \varepsilon^\beta \quad \mbox{for some} \quad \beta \in (0,1).
\end{aligned}
\end{equation*}
Then, we have
\begin{itemize}
 \item[(i)] $\rho_{f^n}, T_{f^n} > C_1,$
 \item[(ii)] $\rho_{f^n} + |U_{f^n}| + T_{f^n} < C_2,$
\end{itemize}
for some generic positive constants $C_1$ and $C_2$.
\end{lemma}

\begin{proof}
Along the backward characteristic defined in \eqref{back}, we find
\begin{align*}\begin{aligned}
&\frac{d}{ds}f^{n+1}(Z^{n+1}(s),s) \\
&\quad= \pa_s f^{n+1}(Z^{n+1}(s),s) + V^{n+1}(s) \cdot \nabla_x f^{n+1}(Z^{n+1}(s),s) \cr
&\qquad + (u^n(X^{n+1}(s),s)-V^{n+1}(s))\cdot \nabla_v f^{n+1}(Z^{n+1}(s),s) \\
&\quad = \rho_{f^n}(X^{n+1}(s),s)\mathcal M(f^n)(Z^{n+1}(s),s) + (3-\rho_{f^n}(X^{n+1}(s),s))f^{n+1}(Z^{n+1}(s),s).
\end{aligned}\end{align*}
We integrate both sides with respect to time to get
\begin{align}\begin{aligned} \label{f_back}
f^{n+1}(z,t) &= e^{\int_0^t(3-\rho_{f^n}(X^{n+1}(s),s))\, ds}f_0(Z^{n+1}(0)) \\
&\quad+ \int_0^t e^{\int_s^t(3-\rho_{f^n}(X^{n+1}(\tau),\tau))\, d\tau}\rho_{f^n}(X^{n+1}(s),s) \mathcal M(f^n)(Z^{n+1}(s),s)\,ds.
\end{aligned}\end{align}
First, it is easy to see that
\[
\rho_{f^k} \leq \int_{\R^3} |f^k|(1+|v|)^{q}(1+|v|)^{-q}\,dv \leq C \|f^k\|_q < C\varepsilon^\beta, \quad 1 \leq k \leq n.
\]
We also have
\begin{align}\label{est_low}
\begin{aligned}
\rho_{f^n} &= \int_{\R^3} f^n \,dv \cr
&\geq \int_{\R^3} e^{\int_0^t (3 - \rho_{f^{n-1}}(X^{n}(s),s))ds}f_0(Z^n(0)) \,dv \\
& \geq  e^{(3-C\varepsilon^\beta)t} \int_{\R^3}f_0(Z^n(0))\,dv\\
 &\geq e^{(3-C\varepsilon^\beta)t}  \varepsilon_1 \int_{\R^3}(1+|V^n(0)|)^{-(q+3+a)}dv \\
&\geq   \varepsilon_1 \int_{\R^3}(1+C(1+|v|))^{{-(q+3+a)}}dv\\
&= C >0,
\end{aligned}
\end{align}
where Lemma \ref{lem_flow} and the assumption on the initial data $f_0$ are used. For the estimate of $U_{f^n}$, we use the lower bound estimate for $\rho_{f^n}$ above to get
\begin{align*}
|U_{f^n}| &\leq \frac{1}{C_{T}} \int_{\R^3}vf^n \,dv \leq \frac{1}{C_{T}} \int_{\R^3}f^n (1+|v|)^{q} (1+|v|)^{1-q}\,dv \leq C_{T}\|f^n\|_q < C_{T}\varepsilon^\beta.
\end{align*}
The upper bound estimate of $T_{f^n}$ can be achieved in a similar way using the estimates above and the lower bound directly follows from Lemma \ref{lem_temp} (i) with \eqref{est_low}.
\end{proof}
In what follows, we show the uniform-in-$n$ boundedness of $f^n$.
\begin{lemma} (Uniform-in-$n$ boundedness of $f^n$) \label{lem_kinet}
Suppose that the initial data $(f_0,\rho_0,u_0)$ satisfies the assumptions of Theorem \ref{thm_main} and $u^n$ satisfies \eqref{n_condition}. Then, there exists a unique solution $f^{n+1}$ to system \eqref{app_main} such that
\begin{equation*}
\sum_{|\nu|\leq 1}\|\nabla^\nu f^{n+1}\|_q < \varepsilon^\beta \quad \mbox{for some} \quad \beta \in (0,1).
\end{equation*}
\end{lemma}
\begin{proof}
\noindent $\bullet$ (Preparatory estimates): Using the upper bound of $\|u^n\|_{L^\infty}$, we obtain from \eqref{B-1-2} that
\[
|V^{n+1}(t)| \geq C_1|v|  - C_2, \quad 0 < C_1 < 1,
\]
for all $0 \leq t \leq T$. It readily gives
\begin{align} \label{B-3}
1 + C_2 + |V^{n+1}(t)| \geq 1+C_1|v| \geq C_1(1+|v|),
\end{align}
for all $0 \leq t \leq T$.
We use the estimate above to find
$$\begin{aligned}
f_0(Z^{n+1}(0)) &= f_0(Z^{n+1}(0))(1 + C_2 + |V^{n+1}(0)|)^q (1 + C_2 + |V^{n+1}(0)|)^{-q} \\
&\leq C f_0(Z^{n+1}(0))(C_2^q + (1+|V^{n+1}(0)|)^q )(1 + C_2 + |V^{n+1}(0)|)^{-q}\cr
&\leq C\|f_0\|_q(1 + C_2 + |V^{n+1}(0)|)^{-q}.
\end{aligned}$$
This together with \eqref{B-3} gives
\begin{align*}
\begin{aligned}
|f_0(Z^{n+1}(0))| \leq  C\|f_0\|_{q} (1+|v|)^{-q},
\end{aligned}
\end{align*}
Similarly, we find
\begin{align} \label{B-7}
\begin{aligned}
& |f^{n+1}(Z^{n+1}(s),s)| \leq C \|f^{n+1}(\cdot,\cdot,s)\|_q (1+|v|)^{-q},\\
& |\nabla_{x,v}f^{n+1}(Z^{n+1}(s),s)| \leq C \| \nabla_{x,v}f^{n+1}(\cdot,\cdot,s)\|_q (1+|v|)^{-q},\\
 &| \nabla_x f_0(Z^{n+1}(0))| \leq C \| \nabla_x f_0\|_q (1+|v|)^{-q},\\
 &| \nabla_v f_0(Z^{n+1}(0))| \leq C \| \nabla_v f_0\|_q (1+|v|)^{-q},\\
 &|\mathcal M(f^n)(Z^{n+1}(s),s)| \leq C\|\mathcal M(f^n)\|_q (1+|v|)^{-q},\\
 & | \nabla_{x,v} \mathcal M(f^n)(Z^{n+1}(s),s)| \leq C \| \nabla_{x,v}\mathcal M(f^n)\|_q (1+|v|)^{-q}.\\
\end{aligned}
\end{align}

\noindent $\bullet$ (Zeroth order estimate): In view of the boundedness of $\rho^{n+1}$ and $\rho_{f^n}$, we get from \eqref{f_back} and the estimates above combined with Lemma \ref{lem_M_bdd} that
\begin{align*}
|f^{n+1}(z,t)| &\leq C |f_0(Z^{n+1}(0))| + C\varepsilon^{\beta}\int_0^t |\mathcal M(f^n)(Z^{n+1}(s),s)|\,ds \\
&\leq C \|f_0\|_{q} (1+|v|)^{-q}+ C \varepsilon^{\beta}\int_0^t \| \mathcal M(f^n)\|_q\,ds \cdot (1+|v|)^{-q} \\
&\leq C  \|f_0\|_{q} (1+|v|)^{-q}+ C \varepsilon^{\beta}\int_0^t \|f^n\|_q \,ds \cdot (1+|v|)^{-q}.
\end{align*}
This readily gives
\begin{equation} \label{C-2}
\sup_{0 \leq t \leq T}\|f^n(\cdot,\cdot,t)\|_q \leq Ce^{C\varepsilon^{\beta}} \|f_0\|_{q},
\end{equation}
for some $C>0$ independent of $n$.
\vspace{0.2cm}

\noindent $\bullet$ (First order estimate): For $j=1,2,3$, we take a partial derivative $\pa_{x_j}$ to the following equation:
\begin{equation} \label{C-3}
\nabla_v \cdot ((u^n-v)f^{n+1}) = (u^n-v) \cdot \nabla_v f^{n+1} - 3f^{n+1},
\end{equation}
then we have
\begin{equation*}
\pa_{x_j}(\nabla_v \cdot ((u^n-v)f^{n+1}))  = \pa_{x_j}u^n\cdot \nabla_v f^{n+1} + (u^n-v)\cdot \nabla_v \pa_{x_j}f^{n+1} -3\pa_{x_j}f^{n+1}.
\end{equation*}
We now take the differential operator $\pa_{x_j}$ to the kinetic equation in \eqref{app_main} and use the estimate above to find
\begin{align*}
\begin{aligned}
&\pa_t \pa_{x_j}f^{n+1} + v \cdot \nabla_x \pa_{x_j}f^{n+1} + (u^n-v) \cdot \nabla_v \pa_{x_j}f^{n+1}\\
&\qquad = (\mathcal M(f^n)-f^{n+1}) \pa_{x_j}\rho_{f^n} + \rho_{f^n}(\pa_{x_j}\mathcal M(f^n) - \pa_{x_j}f^{n+1})  + 3\pa_{x_j}f^{n+1} -  \pa_{x_j}u^n \cdot \nabla_v f^{n+1}.
\end{aligned}
\end{align*}
Then along the characteristic curve $Z^{n+1}(s)$ given in \eqref{back}, we have
\begin{align*}
\frac{d}{dt}\pa_{x_j}f^{n+1}(Z^{n+1}(t),t) &= (3-\rho_{f^n})\pa_{x_j}f^{n+1} + \pa_{x_j}\rho_{f^n}(\mathcal M(f^n)-f^{n+1}) + \rho_{f^n}\pa_{x_j}\mathcal M(f^n) - \pa_{x_j}u^n \cdot \nabla_v f^{n+1}.
\end{align*}
Here and the estimate below, for simplicity we omit the dependence of terms in the right hand sides on $Z^{n+1}(t)$. Then we easily find
\begin{align*}
\begin{aligned}
&\pa_{x_j}f^{n+1}(Z^{n+1}(t),t) \\
&\quad = \pa_{x_j}f_0(Z^{n+1}(0))e^{\int_0^t (3-\rho_{f^n})(X^{n+1}(s),s)\,ds}\\
&\qquad  + \int_0^t \lt( \pa_{x_j}\rho_{f^n}(\mathcal M(f^n)-f^{n+1}) + \rho_{f^n}\pa_{x_j}\mathcal M(f^n) - \pa_{x_j}u^n \cdot \nabla_v f^{n+1}\rt)(Z^{n+1}(s),s) e^{\int_s^t (3-\rho_{f^n})(X^{n+1}(\tau),\tau)\,d\tau}\,ds,
\end{aligned}
\end{align*}
which readily gives
\begin{align*}
\begin{aligned}
|\nabla_x f^{n+1}(z,t)| &\leq C| \pa_{x_j}f_0(Z^{n+1}(0))| \cr
&\quad + C\int_0^t \lt| \lt(\pa_{x_j}\rho_{f^n}(\mathcal M(f^n)-f^{n+1}) + \rho_{f^n}\pa_{x_j}\mathcal M(f^n) - \pa_{x_j}u^n \cdot \nabla_v f^{n+1}\rt)(Z^{n+1}(s),s)\rt| ds.
\end{aligned}
\end{align*}
The terms on the right hand side can be estimated as follows. The estimate of the first term is provided in \eqref{B-7}. We estimate the integrand terms as follows. Using \eqref{A-4-1} and \eqref{B-7}, we have
\begin{equation*}
\lt| \pa_{x_j}\rho_{f^n}(\mathcal M(f^n)-f^{n+1})  \rt| \leq C\|\nabla_x f^n\|_q (\| \mathcal M(f^n)\|_q + \|f^{n+1}\|_q)(1+|v|)^{-q}.
\end{equation*}
Similarly,
\begin{align*}
\lt| \rho_{f^n} \pa_{x_j}\mathcal M(f^n)\rt| \leq C\varepsilon^{\beta}\| \nabla_{x,v}\mathcal M(f^n)\|_q(1+|v|)^{-q}, \quad  \lt| \pa_{x_j}u^n \cdot \nabla_v f^{n+1}\rt| \leq C \| \nabla_v f^{n+1}\|_q (1+|v|)^{-q}.
\end{align*}
Thus, we find
\begin{align} \label{B-10}
\begin{aligned}
\| \nabla_x f^{n+1}(z,t)\|_q &\leq C\| \nabla_x f_0\|_q + C\int_0^t \left( \|\mathcal M(f^n)\|_q + \|f^{n+1}\|_q + \| \nabla_{x,v}\mathcal M(f^n)\|_q + \| \nabla_v f^{n+1}\|_q \right)ds \\
&\leq C\| \nabla_x f_0\|_q + C\int_0^t  \left(\|f^n\|_q + \|f^{n+1}\|_q + \| \nabla_v f^{n+1}\|_q \right)ds \\
&\leq C_T(\|f_0\|_q + \| \nabla_x f_0\|_q) + C\int_0^t  \| \nabla_v f^{n+1}\|_q \,ds.
\end{aligned}
\end{align}
Here, Lemmas \ref{lem_M_bdd} and \ref{lem_macro} together with \eqref{C-2} are used.

We now perform the estimates for $\| \nabla_v f^{n+1}\|_q $ in much the same way as for $\| \nabla_x f^{n+1}\|_q$. We take $\pa_{v_j}, j=1,2,3$ to \eqref{C-3} to have
\begin{equation*}
\begin{aligned}
&\pa_{v_j}(\nabla_v \cdot ((u^n-v)f^{n+1})) = -4\pa_{v_j}f^{n+1} + (u^n-v)\cdot \nabla_v \pa_{v_j}f^{n+1}.
\end{aligned}
\end{equation*}
Then, taking $\pa_{v_j}$ to the kinetic equation in \eqref{app_main} yields
\begin{align*}
\pa_t \pa_{v_j}f^{n+1} + v \cdot \nabla_x \pa_{v_j}f^{n+1} + (u^n-v) \cdot \nabla_v \pa_{v_j}f^{n+1}   =\rho_{f^n}(\pa_{v_j} \mathcal M(f^n) - \pa_{v_j}f^{n+1})-\pa_{x_j} f^{n+1}+ 4\pa_{v_j}f^{n+1}.
\end{align*}
Along the characteristic flow \eqref{back}, it can be rewritten as
\begin{align*}
\frac{d}{dt}\pa_{v_j}f^{n+1}(Z^{n+1}(t),t) = (4-\rho_{f^n}) \pa_{v_j}f^{n+1} + \rho_{f^n}\pa_{v_j}\mathcal M(f^n) - \pa_{x_j}f^{n+1},
\end{align*}
and this readily gives
\begin{align*}
&\pa_{v_j} f^{n+1}(Z^{n+1}(t),t) \\
&\quad = \pa_{v_j}f_0(Z^{n+1}(0))e^{\int_0^t (4-\rho_{f^n})(X^{n+1}(s),s) \,ds} \\
&\qquad + \int_0^t (\rho_{f^n} \pa_{v_j}\mathcal M(f^n) - \pa_{x_j}f^{n+1})(Z^{n+1}(s),s)e^{\int_s^t (4-\rho_{f^n})(X^{n+1}(\tau),\tau)\,d\tau}ds.
\end{align*}
We use the estimate similar to above to have
\begin{align*}
|\nabla_v f^{n+1}(z,t)| & \leq C| \nabla_v f_0(Z^{n+1}(0))| + C\int_0^t (|\nabla_v \mathcal M(f^n)| + |\nabla_x f^{n+1}|)\,ds\\
&\leq C \| \nabla_v f_0\|_q (1+|v|)^{-q} + C \int_0^t \lt( \|f^n\|_q + \| \nabla_x f^{n+1}\|_q\rt) \cdot (1+|v|)^{-q}\,ds,
\end{align*}
which easily gives
\begin{equation} \label{B-11}
\| \nabla_v f^{n+1}(z,t)\|_q  \leq C_T ( \|f_0\|_q + \| \nabla_v f_0\|_q) + C\int_0^t \|\nabla_x f^{n+1}\|_q \,ds.
\end{equation}
Combining \eqref{B-10} and \eqref{B-11} asserts
\[
\| \nabla_{x,v}f^{n+1}(\cdot,\cdot,t)\|_q \leq C_T(\|f_0\|_q + \|\nabla_{x,v}f_0\|_q) + C\int_0^t  \| \nabla_{x,v} f^{n+1}\|_q\,ds,
\]
and Gr\"onwall's lemma yields
\begin{align}\label{B-12}
\begin{aligned}
\| \nabla_{x,v} f^{n+1}(\cdot,\cdot,t)\|_q & \leq C_T(\|f_0\|_q + \|\nabla_{x,v}f_0\|_q).
\end{aligned}
\end{align}
Finally, we conclude from \eqref{C-2} and \eqref{B-12} that
\[
\sum_{|\nu|\leq 1}\|\nabla^\nu f^{n+1}\|_q \leq C(\| f_0\|_q + \| \nabla_{x,v}f_0\|_q) \leq C_T\e < \e^\beta,
\]
where we used our assumption on $\varepsilon$: $C_T \e^{1-\beta}<1$.
\end{proof}
The next lemma show the the uniform-in-$n$ boundedness of the velocity $u^n$. Since the proof is similar to that of \cite{CK}, we postpone it to Appendix \ref{App}.
\begin{lemma} \label{lem_fluid} (Uniform-in-$n$ boundedness of $u^n$)
Suppose that the initial data $(f_0,\rho_0,u_0)$ and $u^n$ satisfy the assumptions in Theorem \ref{thm_main} and \eqref{n_condition}, respectively. Then, there exists a unique solution $u^{n+1}$ to system \eqref{app_main} such that
\begin{equation*}
\|\pa_t u^{n+1}\|_{\mathcal C([0,T];L^2)} + \|\pa_t u^{n+1}\|_{L^2(0,T;H^1)} + \|u^{n+1}\|_{\mathcal C([0,T];H^2)} + \|u^{n+1}\|_{L^2(0,T;H^3)} < \varepsilon^\alpha.
\end{equation*}
\end{lemma}
%
%
%
%
%

\section{Proof of Theorem \ref{thm_main}} \label{sec:4}
\setcounter{equation}{0}
In this section, we first prove that the approximation sequence $(f^n,\rho^n,u^n)$ is a Cauchy sequence. Subsequently, we show that the corresponding limit $(f,\rho,u)$ is the solution to the system \eqref{main}, and moreover it has the desired regularity \eqref{reg}.

\subsection{Construction of Cauchy sequence}
\begin{lemma} \label{lem_cauchy1}
Let $(f^{n},\rho^n,u^n)$ be the solution to system \eqref{app_main}. Then, the following estimate holds:
\begin{align} \label{cauchy-1}
\begin{aligned}
 &\|(f^{n+1}-f^n)(t)\|^2_{q} \leq C \int_0^t \|(f^n-f^{n-1})(s)\|^2_{q}\,ds + C\int_0^t \|(u^n-u^{n-1})(s)\|^2_{H^2}\,ds,
\end{aligned}
\end{align}
where $C>0$ is independent of $n$.
\end{lemma}

\begin{proof}
\noindent  {\bf (Step 1: estimate of $f^{n+1}-f^n$):} We  consider the forward characteristic  $\Bar Z^n(t):=(\Bar X^n(t),\Bar V^n(t)) := (\Bar X^n(t;0,x,v),\Bar V^n(t;0,x,v))$ given by
\begin{align} \label{forward}
\begin{aligned}
&\frac{d}{dt}\Bar X^{n+1}(t) = \Bar V^{n+1}(t),\\
&\frac{d}{dt}\Bar V^{n+1}(t) = u^n(\Bar X^{n+1}(t),t) - \Bar V^{n+1}(t),
\end{aligned}
\end{align}
subject to the initial data
\[
\Bar Z^{n+1}(0) = (x,v) = z.
\]
A computation similar to that for the backward characteristic shows that $f^{n+1}$ can be formulated as follows.
$$\begin{aligned}
f^{n+1}(\bar Z^{n+1}(t),t)& = e^{\int_0^t (3-\rho_{f^n}(\bar X^{n+1}(s),s))\,ds}f_0(z) \\
&\quad + \int_0^t e^{\int_s^t (3-\rho_{f^n}(\bar X^{n+1}(\tau),\tau))\,d\tau}\rho_{f^n}(\bar X^{n+1}(s),s)\mathcal M(f^n)(\bar Z^{n+1}(s),s)\,ds.
\end{aligned}$$
Then, we have
\begin{align*}
&f^{n+1}(\Bar Z^{n+1}(t),t)-f^n(\Bar Z^{n+1}(t),t) \\
&\qquad = f^n(\Bar Z^n(t),t)-f^n(\Bar Z^{n+1}(t),t) + f^{n+1}(\bar Z^{n+1}(t),t) - f^n(\bar Z^n(t),t)\\
&\qquad =  f^n(\Bar Z^n(t),t)-f^n(\Bar Z^{n+1}(t),t) \\
&\qquad \quad +  \left(e^{\int_0^t (3-\rho_{f^n}(\bar X^{n+1}(s),s))\,ds}-e^{\int_0^t (3-\rho_{f^{n-1}}(\bar X^{n}(s),s))\,ds}   \right)f_0(x,v)  \\
&\qquad \quad+ \int_0^t \bigg(e^{\int_s^t (3-\rho_{f^n}(\bar X^{n+1}(\tau),\tau))\,d\tau}\rho_{f^n}(\bar X^{n+1}(s),s)\mathcal M(f^n)(\bar Z^{n+1}(s),s) \\
&\hspace{3cm} - e^{\int_s^t (3-\rho_{f^{n-1}}(\bar X^{n}(\tau),\tau))\,d\tau}\rho_{f^{n-1}}(\bar X^{n}(s),s)\mathcal M(f^{n-1})(\bar Z^{n}(s),s)\bigg)\,ds \\
&\quad =: \mathcal I_1 + \mathcal I_2 + \mathcal I_3,
\end{align*}
where we denote by $\mathcal I_3$ the integral term. \newline

\noindent $\bullet$ (Estimate of $\mathcal I_1$): We easily estimate
\begin{align} \label{E-0}
\mathcal I_1 \leq \| \nabla_{x,v}f^n\|_{q} |\Bar Z^{n+1}(t)-\Bar Z^n(t)| (1+|v|)^{-q} \leq C |\Bar Z^{n+1}(t)-\Bar Z^n(t)|(1+|v|)^{-q}.
\end{align}

\noindent $\bullet$ (Estimate of $\mathcal I_2$): Note that we have the uniform-in-$n$ bounds of $\rho^n$ and $\rho_{f^n}$ thanks to Lemmas \ref{lem_rho} and  \ref{lem_macro_bdd}. Then, the mean value theorem yields
\begin{align*}
\begin{aligned}
&\left| e^{\int_0^t (3-\rho_{f^n}(\bar X^{n+1}(s),s))\,ds}-e^{\int_0^t (3-\rho_{f^{n-1}}(\bar X^{n}(s),s))\,ds}   \right|\\
&\quad \leq \exp\left(\max\Big\{ \int_0^t (3-\rho_{f^n}(\bar X^{n+1}(s),s))\,ds, \int_0^t (3-\rho_{f^{n-1}}(\bar X^{n}(s),s))\,ds \Big\}   \right) \\
&\hspace{3cm} \times \left| \int_0^t \rho_{f^n}(\bar X^{n+1}(s),s) - \rho_{f^{n-1}}(X^n(s),s)\,ds \right| \\
&\quad \leq   C\int_0^t \lt| \rho_{f^n}(\bar X^{n+1}(s),s)-\rho_{f^{n-1}}(\bar X^n(s),s) \rt|ds.
\end{aligned}
\end{align*}
Note that
\begin{align}
\begin{aligned} \label{E-1}
&\int_0^t \lt| \rho_{f^n}(\bar X^{n+1}(s),s)-\rho_{f^{n-1}}(\bar X^n(s),s) \rt|ds \\
&\qquad \leq   \int_0^t \lt|\rho_{f^n}(\bar X^{n+1}(s),s)-\rho_{f^{n-1}}(\bar X^{n+1}(s),s)\rt| + \lt|\rho_{f^{n-1}}(\bar X^{n+1}(s),s)-\rho_{f^{n-1}}(\bar X^n(s),s)\rt|ds \\
&\qquad \leq  \int_0^t \|(f^n-f^{n-1})(s)\|_q  + \|\nabla \rho_{f^{n-1}}\|_{L^\infty}\lt|\bar X^{n+1}(s)-\bar X^n(s)\rt|ds \\
&\qquad \leq C \int_0^t \|(f^n-f^{n-1})(s)\|_q + \lt|\bar X^{n+1}(s)-\bar X^n(s)\rt|ds .
\end{aligned}
\end{align}
Thus, we find
\begin{align} \label{E-2}
\begin{aligned}
\mathcal I_2 \leq C \|f_0\|_{q} \int_0^t  \left(\|(f^n-f^{n-1})(s)\|_q + \lt|(\bar X^{n+1}-\bar X^n)(s)\rt| \right)ds \cdot (1+|v|)^{-q}.
\end{aligned}
\end{align}
\noindent $\bullet$ (Estimate of $\mathcal I_3$): For notational simplicity, we set
$$\begin{aligned}
&\mathcal A_n:= e^{-\int_s^t \rho_{f^{n-1}}(\bar X^n(\tau),\tau)\,d\tau},  \quad \mathcal B_n := \rho_{f^{n-1}}(\bar X^n(s),s), \quad \mbox{and} \quad \mathcal C_n :=  \mathcal M(f^{n-1})(\bar Z^n(s),s).
\end{aligned}$$
Then, we have
\begin{align*}
\mathcal I_3 &= \int_0^t \lt(\mathcal A_{n+1}\mathcal B_{n+1}\mathcal C_{n+1} - \mathcal A_{n}\mathcal B_{n}\mathcal C_{n}\rt)ds \\
&\leq \int_0^t | \mathcal A_{n+1}\mathcal B_{n+1}||\mathcal C_{n+1} - \mathcal C_n| + |\mathcal A_{n+1} \mathcal C_{n}| | \mathcal B_{n+1} - \mathcal B_n| + |\mathcal B_{n}\mathcal C_{n}||\mathcal A_{n+1}-\mathcal A_n| \,ds.
\end{align*}
We give the estimates of $\mathcal C_{n+1} - \mathcal C_n$, $\mathcal B_{n+1} - \mathcal B_n$, and $\mathcal A_{n+1} - \mathcal A_n$, respectively.
\begin{align*}
\begin{aligned}
&|\mathcal C_{n+1}-\mathcal C_n| \\
&\quad \leq | \mathcal M(f^n)(\bar Z^{n+1}(s),s)-\mathcal M(f^n)(\bar Z^n(s),s)| + |\mathcal M(f^n)(\bar Z^n(s),s) - \mathcal M(f^{n-1})(\bar Z^n(s),s)| \\
&\quad \leq |\nabla_{x,v} \mathcal M(f^n) (\theta \Bar Z^{n+1}(s) + (1-\theta)\Bar Z^n(s),s) \cdot (\Bar Z^{n+1}(s)-\Bar Z^n(s))| \\
&\qquad + |\mathcal M(f^n)(\bar Z^n(s),s) - \mathcal M(f^{n-1})(\bar Z^n(s),s)| \\
&\quad \leq C |\Bar Z^{n+1}(s)-\Bar Z^n(s)|  (1+|v|)^{-q} + |\mathcal M(f^n)(\bar Z^n(s),s) - \mathcal M(f^{n-1})(\bar Z^n(s),s)| \\
&\quad \leq C( |\Bar Z^{n+1}(s)-\Bar Z^n(s)|   +  \| (f^n-f^{n-1})(s)\|_q ) (1+|v|)^{-q},
\end{aligned}
\end{align*}
for some $\theta \in [0,1]$. Here, we used Lemmas \ref{lem_lip} and \ref{lem_macro}. We can also get the estimates for $\mathcal A_{n+1} - \mathcal A_n$ and $\mathcal B_{n+1} - \mathcal B_n$ in the same way as \eqref{E-1}.
\begin{align*}
&|\mathcal A_{n+1} - \mathcal A_n| \leq C \int_s^t \|(f^n-f^{n-1})(\tau)\|_q + |\bar X^{n+1}(\tau)-\bar X^n(\tau)| \,d\tau ,\\
&|\mathcal B_{n+1} - \mathcal B_n| \leq C \|(f^n-f^{n-1})(s)\|_q + |\bar X^{n+1}(s)-\bar X^n(s)|.
\end{align*}
Thus, in view of Lemmas \ref{lem_rho}, \ref{lem_macro_bdd}, and the fact that $| \mathcal C_n | \leq C \|f^{n-1}\|_{q} (1+|v|)^{-q}$, which is by Lemma \ref{lem_M_bdd}, we find that
\begin{align} \label{E-3}
\begin{aligned}
\mathcal I_3 \leq C_T \int_0^t \left( \|(f^n-f^{n-1}(s)\|_q+ \lt|\bar Z^{n+1}(s) - \bar Z^n(s)\rt| \right)ds \cdot (1+|v|)^{-q}.
\end{aligned}
\end{align}
We sum up \eqref{E-0}, \eqref{E-2}, and \eqref{E-3} to get
\begin{align}\begin{aligned} \label{E-4}
|(f^{n+1}-f^n)(t)|(1+|v|)^q &\leq C_T \lt|(\bar Z^{n+1}-\bar Z^n)(s)\rt|  \\
&\quad + C_T \int_0^t \left( \|(f^n-f^{n-1})(s)\|_q+\lt|( \bar Z^{n+1} - \bar Z^n)(s)\rt|\right)ds.
\end{aligned}\end{align}

\noindent  {\bf (Step 2: estimate of $\Bar Z^{n+1}-\Bar Z^n$):} We can easily get from \eqref{forward} that
\begin{equation} \label{E-5}
|(\Bar X^{n+1}-\Bar X^n)(t)| \leq \int_0^t \lt|(\Bar V^{n+1} - \Bar V^n)(s)\rt|ds
\end{equation}
and
\[
\Bar V^{n+1}(t) = ve^{-t} + \int_0^t u^n(\bar X^{n+1}(s),s)e^{-(t-s)}\,ds.
\]
Then we have
\begin{align*}
\begin{aligned}
&|\Bar V^{n+1}(t)-\Bar V^n(t)| \leq \int_0^t \lt| u^n(\Bar X^{n+1}(s),s) - u^{n-1}(\Bar X^n(s),s)\rt| e^{-(t-s)}\,ds.
\end{aligned}
\end{align*}
Using the mean value theorem, we have
\begin{align*}
 \lt| u^n(\Bar X^{n+1}(s),s) - u^{n-1}(\Bar X^n(s),s)\rt|  &\leq |u^n(\Bar X^{n+1}(s),s)-u^n(\Bar X^n(s),s)| + |u^n(\Bar X^n(s),s)-u^{n-1}(\Bar X^n(s),s)| \\
&\leq \|\nabla u^n\|_{L^\infty} |\bar X^{n+1}(s)-\bar X^n(s)| + \|(u^n-u^{n-1})(s)\|_{L^\infty} \\
&\leq  \|\nabla u^n\|_{L^\infty} |\bar X^{n+1}(s)-\bar X^n(s)| + C\|(u^n-u^{n-1})(s)\|_{H^2},
\end{align*}
which gives
\begin{align*}
|(\bar V^{n+1}-\bar V^n)(t)| \leq \int_0^t   \|\nabla u^n\|_{L^\infty}\lt|(\bar X^{n+1}-\bar X^n)(s)\rt|ds + C\int_0^t\| (u^n-u^{n-1})(s)\|_{H^2} \, ds.
\end{align*}
This together with \eqref{E-5} gives
\begin{equation*}
\lt| (\bar Z^{n+1} - \bar Z^n)(t) \rt|  \leq C \int_0^t   \| (u^n-u^{n-1})(s)\|_{H^2}\,ds+ \int_0^t (1 + \|\nabla u^n\|_{L^\infty})\lt| (\bar Z^{n+1} - \bar Z^n)(s) \rt| ds.
\end{equation*}
We then use Gr\"onwall's lemma to have
\[
\|(\Bar Z^{n+1}-\Bar Z^n)(t)\|_{L^\infty} \leq C\int_0^t   \| (u^n-u^{n-1})(s)\|_{H^2}\,ds,
\]
where $C>0$ is independent of $n$.

Finally, by combining the above with \eqref{E-4}, we conclude the desired result.
\end{proof}
\begin{lemma} \label{lem_cauchy2}
Let $(f^{n},\rho^{n},u^n)$ be the solution to system \eqref{app_main}. Then we have the following estimate:
\begin{equation} \label{cauchy-2}
\|(\rho^{n+1}-\rho^n)(t)\|^2_{H^2} \leq C_T \int_0^t \|(u^n-u^{n-1})(s)\|^2_{H^2}\,ds.
\end{equation}
\end{lemma}
\begin{proof}
We obtain from the continuity equation in \eqref{app_main} that
\[
\pa_t(\rho^{n+1}-\rho^n) = - u^n \cdot \nabla(\rho^{n+1}-\rho^n) - (u^n-u^{n-1}) \cdot \nabla \rho^n.
\]
\noindent $\bullet$ (Zeroth order estimate): A straightforward computation gives
\begin{align}
\begin{aligned} \label{E-6}
&\frac 12 \frac{d}{dt}\| \rho^{n+1}-\rho^n\|^2_{L^2} \\
&\quad = -\int_{\T^3} (\rho^{n+1}-\rho^n) u^n \cdot \nabla(\rho^{n+1}-\rho^n)\,dx- \int_{\T^3} (\rho^{n+1}-\rho^n) (u^n-u^{n-1})\cdot \nabla \rho^n \,dx  \\
&\quad \leq \|u^n\|_{L^\infty} \| \nabla(\rho^{n+1}-\rho^n)\|_{L^2} \| \rho^{n+1}-\rho^n\|_{L^2} + \| u^n-u^{n-1}\|_{L^6} \| \nabla \rho^n\|_{L^3} \| \rho^{n+1}-\rho^n\|_{L^2} \\
&\quad \leq C\lt(\|\rho^{n+1}-\rho^n\|_{H^1}^2 + \| u^n-u^{n-1} \|_{H^1}^2 \rt),
\end{aligned}
\end{align}
where we used  the Sobolev embedding $H^1(\T^3) \subseteq L^6(\T^3)$ and the Young's inequality for the last inequality. \newline

\noindent $\bullet$ (First order estimate): For $j=1,2,3$, we use H\"older's inequality to have 
\begin{align} \label{E-7}
\begin{aligned}
&\frac 12 \frac{d}{dt} \| \pa_{x_j}(\rho^{n+1}-\rho^n)\|^2_{L^2} \\
&\quad = -\int_{\T^3}\left( \pa_{x_j} u^n \cdot \nabla(\rho^{n+1}-\rho^n) + u^n \cdot \nabla \pa_{x_j}(\rho^{n+1}-\rho^n)\right)\pa_{x_j}(\rho^{n+1}-\rho^n)\,dx \\
&\qquad - \int_{\T^3} \left( \pa_{x_j}(u^n-u^{n-1})\cdot \nabla \rho^n - (u^n-u^{n-1})\cdot \nabla \pa_{x_j} \rho^n \right) \pa_{x_j}(\rho^{n+1}-\rho^n)\,dx \\
&\quad \leq C\| \pa_{x_j} u^n\|_{L^\infty} \| \nabla(\rho^{n+1}-\rho^n)\|_{L^2}^2 + \| \nabla u^n\|_{L^\infty} \| \pa_{x_j}(\rho^{n+1}-\rho^n)\|_{L^2}^2 \\
&\qquad + \left( \| \pa_{x_j}(u^n-u^{n-1})\|_{L^2} \| \nabla \rho^n\|_{L^\infty} + \|u^n- u^{n-1}\|_{L^6} \| \nabla \pa_{x_j} \rho^n\|_{L^3} \right) \| \pa_{x_j}(\rho^{n+1}-\rho^n)\|_{L^2}\\
&\quad \leq C\lt( \| \rho^{n+1}-\rho^n\|^2_{H^1} + \| u^n- u^{n-1} \|^2_{H^1}\rt).
\end{aligned}
\end{align}

\noindent $\bullet$ (Second order estimate): Similarly, for $i,j=1,2,3$, we obtain 
\begin{align} \label{E-8}
\begin{aligned}
&\frac{1}{2}\frac{d}{dt}\|\pa_{x_i}\pa_{x_j}(\rho^{n+1}-\rho^n)\|_{L^2}^2 \\
&\quad = - \int_{\T^3}\pa_{x_i}\pa_{x_j}(\rho^{n+1}-\rho^n)(\pa_{x_i}\pa_{x_j} u^n \cdot(\rho^{n+1}-\rho^n) +\pa_{x_j} u^n \cdot \nabla \pa_{x_i}(\rho^{n+1}-\rho^n))\,dx \\
&\qquad- \int_{\T^3}\pa_{x_i}\pa_{x_j}(\rho^{n+1}-\rho^n)(\pa_{x_i} u^n \cdot \nabla \pa_{x_j}(\rho^{n+1}-\rho^n) + u^n \cdot \nabla \pa_{x_i}\pa_{x_j}(\rho^{n+1}-\rho^n))\,dx\\
&\qquad- \int_{\T^3} \pa_{x_i}\pa_{x_j}(\rho^{n+1}-\rho^n)(\pa_{x_i}\pa_{x_j}(u^n-u^{n-1})\cdot \nabla \rho^n + \pa_{x_j}(u^n-u^{n-1})\cdot \nabla \pa_{x_i} \rho^n)\,dx\\
&\qquad- \int_{\T^3}\pa_{x_i}\pa_{x_j}(\rho^{n+1}-\rho^n)(\pa_{x_i}(u^n-u^{n-1})\cdot \nabla \pa_{x_j} \rho^n + (u^n-u^{n-1})\cdot \nabla \pa_{x_i}\pa_{x_j} \rho^n)\,dx \\
&\quad \leq C(\|\rho^{n+1}-\rho^n\|_{H^2}^2 + \|u^n-u^{n-1}\|_{H^2}^2).
\end{aligned}
\end{align}
Then, the conclusion follows from the summation of \eqref{E-6}, \eqref{E-7}, and \eqref{E-8}.
\end{proof}
\begin{lemma} \label{lem_cauchy3}
Let $(f^{n},\rho^{n},u^n)$ be the solution to system \eqref{app_main}. Then we have the following estimate:
\begin{align}\begin{aligned} \label{cauchy-3}
&\|(u^{n+1}-u^n)(t)\|_{L^2}^2  + \int_0^t \int_{\T^3} \| \nabla (u^{n+1}-u^n)(s) \|_{L^2}^2 \,dxds \\
\quad & \leq C_T \int_0^t \|(u^n-u^{n-1})(s)\|^2_{H^1} + \|(\rho^{n+1}-\rho^n)(s)\|^2_{H^2} + \|(f^{n+1}-f^n)(s)\|^2_{q}\,ds.
\end{aligned}\end{align}
\end{lemma}
\begin{proof}
We first use $\eqref{app_main}_3$ to find
\begin{align} \label{E-9}
\begin{aligned}
&\rho^n \pa_t(u^{n+1} - u^n) \\
&\quad = -\rho^n u^{n-1} \cdot \nabla(u^{n+1}-u^n) + \Delta(u^{n+1}-u^n) - \nabla(p^{n+1}-p^n) - (\rho^{n+1}-\rho^n) \pa_t u^{n+1}  \\
&\qquad - (\rho^{n+1}-\rho^n)u^n \cdot \nabla u^{n+1} - \rho^n(u^n- u^{n-1})\cdot \nabla u^{n+1} - (\rho^{n+1}-\rho^n) \int_{\R^3}(u^n-v)f^{n+1}dv \\
&\qquad -\rho^n \int_{\R^3}(u^n-u^{n-1})f^{n+1}dv - \rho^n \int_{\R^3}(u^n-v)(f^{n+1}-f^n)\,dv.
\end{aligned}
\end{align}
Then, we have
\begin{align*}
&\frac 12 \frac{d}{dt} \int_{\T^3} \rho^n |u^{n+1}-u^n|^2 \,dx \\
&\quad= \frac 12 \int_{\T^3} \pa_t \rho^n | u^{n+1}-u^n|^2\,dx + \int_{\T^3} \rho^n(u^{n+1}-u^n) \cdot \pa_t(u^{n+1}-u^n)\,dx\\
&\quad = \frac 12 \int_{\T^3} \pa_t \rho^n | u^{n+1}-u^n|^2\,dx + \int_{\T^3}(u^{n+1}-u^n) \cdot \bigg( -\rho^n u^{n-1} \cdot \nabla(u^{n+1}-u^n) \\
&\qquad + \Delta(u^{n+1}-u^n) - \nabla(p^{n+1}-p^n) - (\rho^{n+1}-\rho^n)\pa_t u^{n+1} - (\rho^{n+1}-\rho^n)u^n \cdot \nabla u^{n+1} \\
&\qquad + \rho^n(u^n-u^{n-1}) \cdot \nabla u^{n+1} - (\rho^{n+1}-\rho^n) \int_{\R^3}(u^n-v)f^{n+1}\,dv - \rho^n \int_{\R^3}(u^n-u^{n-1})f^{n+1}\,dv \\
&\qquad - \rho^n \int_{\R^3} (u^n-v)(f^{n+1}-f^n)\,dv \bigg) \,dx \\
&\quad =: \sum_{i=1}^{10} \mathcal J_i.
\end{align*}
The estimates of each term $\mathcal J_i, i=1,\cdots,10$ are given as follows.
\begin{align*}
& \mathcal J_1 = -\frac 12 \int_{\T^3} (u^{n-1} \cdot \nabla \rho^n)|u^{n+1}-u^n|^2\,dx \lesssim \| u^{n+1}-u^n\|_{L^2}^2,\\
& \mathcal J_2 \lesssim \| u^{n+1}-u^n\|_{L^2} \| \nabla(u^{n+1}-u^n)\|_{L^2} \lesssim \| u^{n+1}-u^n\|_{L^2}^2 + \| \nabla(u^{n+1}-u^n)\|_{L^2}^2, \\
& \mathcal J_3 = - \| \nabla(u^{n+1}-u^n)\|_{L^2}^2, \qquad \mathcal J_4  = 0, \\
& \mathcal J_5 \leq \|u^{n+1}-u^n\|_{L^2} \| \pa_t u^{n+1}\|_{L^2 }\| \rho^{n+1} - \rho^n \|_{L^\infty} \lesssim \|u^{n+1}-u^n\|_{L^2}^2+ \| \rho^{n+1} - \rho^n \|_{H^2}^2,\\
& \mathcal J_6 \leq \|u^n\|_{L^6}\| \nabla u^{n+1}\|_{L^6}\| \rho^{n+1}-\rho^n\|_{L^6} \|u^{n+1}-u^n\|_{L^2} \lesssim \|\rho^{n+1}-\rho^n\|_{H^1}^2 + \|u^{n+1}-u^n\|_{L^2}^2,\\
& \mathcal J_7 \leq \| \rho^n\|_{L^\infty} \| u^{n+1}-u^n\|_{L^6} \| u^n-u^{n-1}\|_{L^2} \| \nabla u^{n+1}\|_{L^3} \lesssim \|\nabla(u^{n+1}-u^n)\|_{L^2}^2+   \|u^n-u^{n-1}\|_{L^2}^2,\\
& \mathcal J_8 \leq \|u^{n+1}-u^n\|_{L^2} \| \rho^{n+1}-\rho^n\|_{L^6} \lt\| \int_{\R^3}(u^n-v)f^{n+1}\,dv \rt\|_{L^3} \lesssim \| u^{n+1}-u^n\|_{L^2}^2 + \| \rho^{n+1}-\rho^n\|_{H^1}^2.
\end{align*}
Here, the last term was estimated as follows:
\begin{align*}
\lt\| \int_{\R^3}(u^n-v)f^{n+1}\,dv \rt\|_{L^3} &\lesssim \lt\| \int_{\R^3}(u^n-v)f^{n+1}\,dv \rt\|_{L^\infty} \\
&\leq \lt\| \int_{\R^3} u^n f^{n+1}\,dv\rt\|_{L^\infty} + \lt\|\int_{\R^3} vf^{n+1}\,dv \rt\|_{L^\infty} \\
&\leq \|u^n\|_{L^\infty}\|f^{n+1}\|_q + \|f^{n+1}\|_q \leq C.
\end{align*}
Similarly, $\mathcal J_9$ and $\mathcal J_{10}$ can be estimated as follows.
\begin{align*}
 \mathcal J_9 &\leq \| \rho^n\|_{L^\infty} \|u^{n+1}-u^n\|_{L^2}\|u^n- u^{n-1}\|_{L^6} \lt\|\int_{\R^3}f^{n+1}\,dv  \rt\|_{L^3} \\
&  \lesssim \|u^{n+1}-u^n\|_{L^2}^2 + \|u^n-u^{n-1}\|_{H^1}^2,
\end{align*}
\begin{align*}
 \mathcal J_{10} &\leq \| \rho^n\|_{L^\infty} \| u^{n+1}-u^n\|_{L^2} \lt\| \int_{\R^3}(u^n-v)(f^{n+1}-f^n)\,dv \rt\|_{L^2} \\
& \leq \| \rho^n\|_{L^\infty} \| u^{n+1}-u^n\|_{L^2}  ( \|u^n\|_{L^\infty}+1)\|f^{n+1}-f^n\|_q \\
&\lesssim \|u^{n+1}-u^n\|_{L^2}^2 + \| f^{n+1}-f^n\|_q^2.
\end{align*}
We sum up the estimates above and integrate from 0 to $t$ to get
\begin{align*}
&\int_{\T^3}\rho^n |u^{n+1}-u^n|^2\,dx +\int_0^t \| \nabla(u^{n+1}-u^n)(s)\|_{L^2}^2\,ds \\
&\qquad \leq C_T \int_0^t \|(u^n-u^{n-1})(s)\|^2_{H^1} + \|(\rho^{n+1}-\rho^n)(s)\|^2_{H^2} + \|(f^{n+1}-f^n)(s)\|^2_{q}\,ds.
\end{align*}
Finally, the conclusion follows in view of Lemma \ref{lem_rho}.
\end{proof}
\begin{lemma} \label{lem_cauchy4}
Let $(f^{n},\rho^{n},u^n)$ be the solution to system \eqref{app_main}. Then we have the following estimate:
\begin{align} \label{cauchy-4}
\begin{aligned}
&\|\nabla(u^{n+1}-u^n)(t)\|^2_{L^2} + \int_0^t \|\pa_s(u^{n+1}-u^n)(s)\|^2_{L^2}\,ds \\
&\quad \leq C_T \int_0^t \|(\rho^{n+1}-\rho^n)(s)\|^2_{H^2} + \|(u^n-u^{n-1})(s)\|^2_{H^1} + \|(f^{n+1}-f^n)(s)\|^2_{q}\,ds.
\end{aligned}
\end{align}
\end{lemma}
\begin{proof}
We take an inner product of both sides of \eqref{E-9} with $\pa_t(u^{n+1}-u^n)$ and integrate it over $\T^3$  to find
\begin{align*}
&\int_{\T^3} \rho^n |\pa_t(u^{n+1}-u^n)|^2\,dx + \frac{1}{2}\frac{d}{dt}\int_{\T^3}| \nabla (u^{n+1}-u^n)|^2\,dx \\
&\quad = -\int_{\T^3} \pa_t(u^{n+1}-u^n) \cdot \bigg(\rho^n u^{n-1}\cdot \nabla(u^{n+1}-u^n) + (\rho^{n+1}-\rho^n) \pa_t u^{n+1} \\
&\qquad + (\rho^{n+1}-\rho^n)u^n \cdot \nabla u^{n+1} + \rho^n(u^n - u^{n-1})\cdot \nabla u^{n+1} + (\rho^{n+1}-\rho^n) \int_{\R^3} (u^n-v)f^{n+1}\,dv \\
&\qquad + \rho^n \int_{\R^3}(u^n-u^{n-1}) f^{n+1}\,dv + \rho^n \int_{\R^3}(u^n-v)(f^{n+1}-f^n) \,dv \bigg)\,dx \\
&\quad =: \sum_{i=1}^7 \mathcal K_i.
\end{align*}
We can derive the estimates similar to those in Lemma \ref{lem_cauchy3}.
\begin{align*}
\mathcal K_1 &\leq \| \rho^n\|_{L^\infty} \|u^{n-1}\|_{L^\infty} \| \pa_t(u^{n+1}-u^n)\|_{L^2} \| \nabla(u^{n+1}-u^n)\|_{L^2}\\
  &\lesssim \| \pa_t(u^{n+1}-u^n)\|_{L^2}^2 + \| \nabla(u^{n+1}-u^n)\|_{L^2}^2,\\
\mathcal K_2 &\leq \| \rho^{n+1}-\rho^n\|_{L^\infty} \| \pa_t(u^{n+1}-u^n)\|_{L^2} \| \pa_t u^{n+1}\|_{L^2} \\
  &\lesssim \| \pa_t(u^{n+1}-u^n)\|_{L^2}^2 + \| \rho^{n+1}-\rho^n\|_{H^2}^2,\\
\mathcal K_3 &\leq \| \rho^{n+1}-\rho^n\|_{L^6} \|u^n\|_{L^\infty} \| \nabla u^{n+1}\|_{L^3} \| \pa_t(u^{n+1}-u^n)\|_{L^2} \\
  &\lesssim \| \pa_t(u^{n+1}-u^n)\|_{L^2}^2 + \| \rho^{n+1}-\rho^n\|_{H^1}^2,\\
\mathcal K_4 &\leq \| \rho^n\|_{L^\infty} \|u^n-u^{n-1}\|_{L^6} \| \nabla u^{n+1}\|_{L^3} \| \pa_t(u^{n+1}-u^n)\|_{L^2} \\
  &\lesssim \| \pa_t(u^{n+1}-u^n)\|_{L^2}^2 + \|u^n-u^{n-1}\|_{H^1}^2,\\
\mathcal K_5 &\leq \|\rho^{n+1}-\rho^n\|_{L^6} \| \pa_t(u^{n+1}-u^n)\|_{L^2} \lt\| \int_{\R^3}(u^n-v)f^{n+1}\,dv \rt\|_{L^3} \\ &\lesssim \| \pa_t(u^{n+1}-u^n)\|_{L^2}^2 + \| \rho^{n+1}-\rho^n\|_{H^1}^2,\\
\mathcal K_6 &\leq \| \rho^n\|_{L^\infty} \| \pa_t(u^{n+1}-u^n)\|_{L^2} \|u^n-u^{n-1}\|_{L^6}\lt\| \int_{\R^3}f^{n+1}\,dv \rt\|_{L^3} \\
  &\lesssim \| \pa_t(u^{n+1}-u^n)\|_{L^2}^2 + \|u^n-u^{n-1}\|_{H^1}^2, \\
\mathcal K_7 &\leq \| \rho^n\|_{L^\infty} \| \pa_t(u^{n+1}-u^n)\|_{L^2} \lt\| \int_{\R^3}(u^n-v)(f^{n+1}-f^n)\,dv \rt\|_{L^2}\\ &\lesssim \| \pa_t(u^{n+1}-u^n)\|_{L^2}^2 + \|f^{n+1}-f^n\|_q^2.
\end{align*}
So, we have
\begin{align*}
&\int_{\T^3} \rho^n |\pa_t(u^{n+1}-u^n)|^2 \,dx + \frac{1}{2}\frac{d}{dt}\int_{\T^3}| \nabla (u^{n+1}-u^n)|^2 \,dx  \\
&\qquad \lesssim \|\rho^{n+1}-\rho^n\|^2_{H^2} + \|u^n-u^{n-1}\|^2_{H^1} + \|f^{n+1}-f^n\|^2_{q}.
\end{align*}
Finally, we take an integration from 0 to $t$ and use Lemma \ref{lem_rho} to obtain the desired result.
\end{proof}
\begin{lemma} \label{lem_cauchy5}
Let $(f^{n},\rho^{n},u^n)$ be the solution to system \eqref{app_main}. Then we have the following estimate:
\begin{align}
\begin{aligned} \label{cauchy-5}
&\| \nabla^2(u^{n+1}-u^n)\|^2_{L^2} + \| \nabla(p^{n+1}-p^n)\|^2_{L^2} \\
& \quad \leq C( \|\rho^{n+1}-\rho^n\|^2_{H^2} + \|\pa_t(u^{n+1}-u^n)\|^2_{L^2} + \| \nabla(u^n-u^{n-1})\|^2_{L^2} )\\
&\qquad +  \| \nabla(u^{n+1}-u^n)\|^2_{L^2} + \|f^{n+1}-f^n\|^2_{q}).
\end{aligned}
\end{align}
\end{lemma}
\begin{proof}
We obtain from \eqref{C-5-0} and \eqref{C-5} that
\begin{align*}
& \| \nabla^2(u^{n+1}-u^n)\|_{L^2}^2 + \| \nabla(p^{n+1}-p^n)\|_{L^2}^2 \\
&\quad \lesssim \|(\rho^{n+1}-\rho^n)\pa_t u^{n+1}\|_{L^2}^2 + \| \rho^n \pa_t (u^{n+1}-u^n)\|_{L^2}^2 + \|(\rho^{n+1}-\rho^n) u^n \cdot \nabla u^{n+1}\|_{L^2}^2 \\
&\qquad + \| \rho^n(u^n-u^{n-1})\cdot \nabla u^{n+1}\|_{L^2}^2 + \| \rho^n u^{n-1} \cdot \nabla(u^{n+1}-u^n)\|_{L^2}^2 \\
&\qquad + \left\|(\rho^{n+1}-\rho^n) \int_{\R^3}(u^n-v)f^{n+1}\,dv \right\|_{L^2}^2 + \left\|  \rho^n \int_{\R^3}(u^n-u^{n-1})f^{n+1}\,dv \right\|_{L^2}^2 \\
&\qquad + \left\| \rho^n \int_{\R^3}(u^n-v)(f^{n+1}-f^n)\,dv      \right\|_{L^2}^2\\
&\quad =: \sum_{i=1}^8 \mathcal L_i.
\end{align*}
The estimates for $\mathcal L_i$ can be done in the way similar to Lemma \ref{cauchy-3} and we omit the details.
\begin{align*}
&\mathcal L_1, \,\mathcal L_3, \,\mathcal L_6 \lesssim \|\rho^{n+1}-\rho^n\|_{H^2}^2, \quad \mathcal L_2 \lesssim \|\pa_t(u^{n+1}-u^n)\|_{L^2}^2, \\
&\mathcal L_4, \, \mathcal L_7 \lesssim \| \nabla (u^n-u^{n-1})\|_{L^2}^2, \quad \mathcal L_5 \lesssim \| \nabla(u^n-u^{n-1})\|_{L^2}^2,\quad    \mathcal L_8 \lesssim \|f^{n+1}-f^n\|_q^2
\end{align*}
\end{proof}
%
%
%
%
\subsection{Proof of Theorem \ref{thm_main}}
We are now ready to prove the existence and uniqueness of solution to \eqref{main}. \newline

\noindent $\bullet$ (Existence):
We sum up \eqref{cauchy-1}, \eqref{cauchy-3}, and \eqref{cauchy-4} using \eqref{cauchy-2} to derive
\begin{align}
\begin{aligned} \label{D-1}
&\|f^{n+1}-f^n\|_{q}^2  + \|u^{n+1}-u^n\|^2_{H^1} + \int_0^t \| \nabla(u^{n+1}-u^n)(s)\|^2_{L^2} + \|\pa_s(u^{n+1}-u^n)(s)\|^2_{L^2}\,ds \\
&\quad \leq C \int_0^t \|(u^n-u^{n-1})(s)\|^2_{H^1} + \|(f^n-f^{n-1})(s)\|^2_{q}\,ds + C \int_0^t \int_0^s \|(u^n-u^{n-1})(\tau)\|^2_{H^2}\,d\tau ds.
\end{aligned}
\end{align}
We integrate both sides of \eqref{cauchy-5} from 0 to $t$ and use \eqref{cauchy-2} again to have
\begin{align}
\begin{aligned} \label{D-2}
&\int_0^t \| \nabla^2(u^{n+1}-u^n)(s)\|^2_{L^2} \cr
&\quad \leq C\int_0^t \|(u^n-u^{n-1})(s)\|^2_{H^1}\,ds + \|\nabla(u^{n+1}-u^n)(s)\|^2_{L^2}\,ds \\
& \qquad  + C\int_0^t  \|(\pa_s(u^{n+1}-u^n)(s)\|^2_{L^2} \,ds  + C\int_0^t \int_0^s \|(u^n-u^{n-1})(\tau)\|^2_{H^2} \,d\tau ds.
\end{aligned}
\end{align}
Combining \eqref{D-1} and \eqref{D-2} yields
$$\begin{aligned}
&\|f^{n+1}-f^n\|_{q}^2  + \|u^{n+1}-u^n\|^2_{H^1} + \int_0^t \| \nabla(u^{n+1}-u^n)(s)\|^2_{H^1}\,ds \\
&\quad \leq C \int_0^t \lt(\|(f^{n+1}-f^n)(s)\|_{q}^2 +  \| (u^n-u^{n-1})(s)\|^2_{H^1}\rt)ds  +C\int_0^t \lt( \int_0^s \| \nabla(u^n-u^{n-1})(\tau)\|^2_{H^1}\,d\tau \rt)\,ds.
\end{aligned}$$
Using the induction argument, we have
\begin{align*}
&\|f^{n+1}-f^n\|^2_{\mathcal C([0,T];L_q^\infty)} + \|u^{n+1}-u^n\|_{\mathcal C([0,T];H^1)}^2 + \| \nabla(u^{n+1}-u^n)\|_{L^2(0,T;H^1)}^2 \\
& \qquad \leq \frac{C(T)^{n+1}}{n!},
\end{align*}
which yields that there exist the limit function $(f,\rho,u)$ such that
$$\begin{aligned}
&f^n \to f \quad \mbox{in}~\mathcal C([0,T];L^\infty_{q}(\T^3 \times \R^3)), \quad \rho^n \to \rho \quad \mbox{in}~ \mathcal C([0,T];H^2(\T^3)) \\
&u^n \to u \quad \mbox{in}~\mathcal C([0,T];H^1(\T^3)) \cap L^2([0,T];H^2(\T^3)).
\end{aligned}$$
On the other hand, in view of uniform-in-$n$ boundedness(Lemmas \ref{lem_kinet} and \ref{lem_fluid}),  Banach-Alaoglu theorem yields that there exists a subsequence $(f^{n_k},\rho^{n_k},u^{n_k})$ and its weak limit $(\widetilde f, \widetilde \rho, \widetilde u)$ such that
$$\begin{aligned}
&f^{n_k }\rightharpoonup \widetilde f  \quad \mbox{weakly}-\ast \quad \mbox{in}~\mathcal C([0,T];W^{1,\infty}_{q}(\T^3 \times \R^3)), \\
& \rho^{n_k} \rightharpoonup \widetilde \rho \quad \mbox{in}~\mathcal C([0,T];H^3(\T^3)),\quad \mbox{and} \\
&u^{n_k} \rightharpoonup \widetilde u \quad \mbox{in}~\mathcal C([0,T];H^2(\T^3)) \cap L^2(0,T;H^3(\T^3)).
\end{aligned}$$
Then, we have
\begin{align*}
&f \equiv \widetilde f \quad \mbox{in}~\mathcal C([0,T];L^\infty_{q}(\T^3 \times \R^3)), \quad \rho \equiv \widetilde \rho \quad \mbox{in}~ \mathcal C([0,T];H^2(\T^3)),\quad \mbox{and} \\
&u \equiv \widetilde u  \quad \mbox{in}~\mathcal C([0,T];H^1(\T^3)) \cap L^2([0,T];H^2(\T^3)),
\end{align*}
which is due to the uniqueness of weak limit. We now claim that indeed
\begin{align*}
&f \equiv \widetilde f \quad \mbox{in}~\mathcal C([0,T];W_q^{1,\infty}(\T^3 \times \R^3)), \quad \rho \equiv \widetilde \rho \quad \mbox{in}~ \mathcal C([0,T];H^3(\T^3)),\quad \mbox{and} \\
&u \equiv \widetilde u  \quad \mbox{in}~\mathcal C([0,T];H^2(\T^3)) \cap L^2([0,T];H^3(\T^3)).
\end{align*}
To this end,
\begin{align*}
\left| \int_{\T^3 \times \R^3}\pa(f-\widetilde f)\phi \,dxdv\right| &= \left|\int_{\T^3 \times \R^3}(f-\widetilde f)\pa \phi \, dxdv \right| \\
&\leq \|f-\widetilde f\|_{\mathcal C(0,T;L^\infty)}\left| \int_{\T^3 \times \R^3}|\pa \phi|\,dxdv \right| = 0,
\end{align*}
for $\forall \phi \in \mathcal C_c^\infty(\T^3 \times \R^3)$. Thus, we have $\pa f = \pa \widetilde f$ a.e. in $\T^3 \times \R^3$ and the first assertion holds. Similarly,
\begin{align*}
\left| \int_{\T^3}\nabla^2 (u-\widetilde u)\phi \, dx\right| &= \left|\int_{\T^3}\nabla(u-\widetilde u)\nabla \phi \,dx\right| \\
&\leq \| \nabla(u-\widetilde u)\|_{\mathcal C(0,T;L^2)} \|\nabla \phi\|_{L^2} = 0, \quad \forall \phi \in \mathcal C_c(\T^3),
\end{align*}
which implies $u \equiv \widetilde u  ~ \mbox{in}~\mathcal C([0,T];H^2(\T^3))$. Moreover, it gives
\begin{align*}
\left|\int_0^T \int_{\T^3}\nabla^3(u-\widetilde u)\phi\,dxds \right| &= \left|\int_0^T \int_{\T^3}\nabla^2(u-\widetilde u)\nabla \phi \,dxds\right| \\
&\leq \left|\int_0^T \| \nabla^2(u-\widetilde u)\|_{L^2} \|\nabla \phi\|_{L^2}\, ds \right| = 0,
\end{align*}
which yields $\nabla^3 u = \nabla^3 \widetilde u$, a.e. in $\T^3 \times [0,T]$ and the third assertion holds. The second assertion can be proved in the same way, and we omit the proof. It now remains to prove the strong convergence of local Maxwellian $\mathcal M(f^n) \to \mathcal M(f)$ as $n \to \infty$, and it suffices to show the strong convergence of the macroscopic fields $(\rho_{f^n},U_{f^n},T_{f^n}) \to (\rho_f,U_f,T_f)$. First, note that
\begin{align*}
|\rho_{f^n}-\rho_f| \leq \int_{\R^3}|f^n-f|\,dv \lesssim \|f^n-f\|_q \to 0 \quad \mbox{as}~n \to \infty.
\end{align*}
Using this convergence and Lemma \ref{lem_macro_bdd}, we have
\begin{align*}
|U_{f^n}-U_f| &= \left| \frac{1}{\rho_{f^n}}\int_{\R^3}vf^n\,dv -  \frac{1}{\rho_{f}}\int_{\R^3}vf \,dv \right| \\
&\leq \frac{1}{\rho_{f^n}} \left| \int_{\R^3}v(f^n-f)\,dv     \right| + \left| \frac{1}{\rho_{f^n}}-\frac{1}{\rho_f}  \right|   \left|  \int_{\R^3}vf\,dv    \right| \\
&\lesssim \|f^n-f\|_q + |\rho_{f^n}-\rho_f| \to 0 \quad \mbox{as}~n \to \infty.
\end{align*}
In the similar way, we get
$$\begin{aligned}
&|T_{f^n}-T_f| \\
&\quad \leq \left| \frac{1}{3\rho_{f^n}} - \frac{1}{3\rho_f}  \right| \int_{\R^3}|v-U_{f^n}|^2 f^n dv + \frac{1}{3\rho_f}\left|\int_{\R^3}|v-U_{f^n}|^2 f^n  dv - \int_{\R^3} |v-U_f|^2 f dv  \right| \\
&\quad \lesssim  |\rho_{f^n}-\rho_f| + \int_{\R^3}|v-U_{f^n}|^2 |f^n-f|\,dv + \int_{\R^3}\left| |v-U_{f^n}|^2 - |v-U_f|^2  \right| f \,dv \\
&\quad \lesssim  |\rho_{f^n}-\rho_f| + \int_{\R^3}(1+|v|)^2 |f^n-f|\,dv + |U_{f^n} - U_f| \int_{\R^3} (1+ |v|)f\,dv \\
&\quad \lesssim   |\rho_{f^n}-\rho_f| + \|f^n-f\|_q + |U_{f^n}-U_f| \to 0 \quad \mbox{as}~n \to \infty.
\end{aligned}$$
\noindent $\bullet$ (Uniqueness): Let $(f_1,\rho_1,u_1)$ and $(f_2,\rho_2,u_2)$ be the solutions to system \eqref{main}-\eqref{ini_main} with the same initial data $(f_0,\rho_0,u_0)$. Using the argument similar to that in a series of Lemmas in this section, we can prove that the functional $\Delta(t) := \|f_1-f_2\|_{q}^2 + \|\rho_1-\rho_2\|_{H^2}^2 + \|u_1-u_2\|_{H^1}^2$ satisfies the following Gr\"onwall's inequality:
\[
\Delta(t) \lesssim \int_0^t \Delta(s)\,ds, \quad \Delta(0) = 0,
\]
which readily gives that
\begin{align*}
&f_1 \equiv f_2 \quad \mbox{in}~\mathcal C([0,T];L_q^\infty(\T^3 \times \R^3)), \quad \rho_1 \equiv \rho_2 \quad \mbox{in}~\mathcal C([0,T];H^2(\T^3)), \quad \mbox{and}\\
&u_1 \equiv u_2 \quad \mbox{in}~\mathcal C([0,T];H^1(\T^3)) \cap L^2([0,T];H^2(\T^3)).
\end{align*}
The same result for the higher regularity can be shown in the exactly same way as in the existence part.

%
%
%
%

\appendix
\section{Proof of Lemma \ref{lem_fluid} }\label{App}
\setcounter{equation}{0}
 We divide the proof into four steps. In each step, we will show the followings:\\

\noindent $\bullet$ In Step A, we provide the $H^1$-estimate of $u^{n+1}$:
\begin{equation*}
\|u^{n+1}\|_{\mathcal C([0,T];L^2)} + \|\nabla u^{n+1}\|_{L^2(0,T;L^2)} < \frac{\varepsilon^\alpha}{10}.
\end{equation*}
\noindent $\bullet$ In Step B, we show the first order and $\dot{H}^2(\T^3)$ estimates of $u^{n+1}$:
\begin{equation*}
\|\pa_t u^{n+1}\|_{L^2(0,T;L^2)} + \|\nabla^2 u^{n+1}\|_{L^2(0,T;L^2)} + \| \nabla u^{n+1}\|_{\mathcal C([0,T];L^2)} + \| \nabla p^{n+1}\|_{L^2(0,T;L^2)} < \frac{\varepsilon^\alpha}{10}.
\end{equation*}
\noindent $\bullet$ In Step C, we present the $H^1$-estimate of $\pa_t u^{n+1}$:
\begin{equation*}
\|\pa_t u^{n+1}\|_{\mathcal C(0,T];L^2)} + \| \nabla \pa_t u^{n+1}\|_{L^2(0,T;L^2)} < \varepsilon^{\alpha^*}< \varepsilon^\alpha,
\end{equation*}
where $\alpha < \alpha^* < \min \{\beta, (3\alpha)/2 \}$. \newline
\noindent $\bullet$ In Step D, we finally provide the high-order estimate of $u^{n+1}$:
\begin{equation*}
\| \nabla^2 u^{n+1}\|_{\mathcal C([0,T];L^2)} + \| \nabla^3 u^{n+1} \|_{L^2(0,T;L^2)} + \| \nabla p^{n+1}\|_{\mathcal C([0,T];L^2)} + \| \nabla^2 p^{n+1}\|_{L^2(0,T;L^2)} < \frac{\varepsilon^\alpha}{10}.
\end{equation*}

\noindent $\bullet$ (Step A): We take an inner product of both sides of $\eqref{app_main}_3$ with $u^{n+1}$ and integrate it over $\T^3$  to find
\begin{align*}
&\frac{1}{2}\frac{d}{dt}\int_{\T^3}\rho^{n+1}|u^{n+1}|^2 \,dx + \int_{\T^3} |\nabla u^{n+1}|^2 \,dx \\
&\qquad = \frac{1}{2}\int_{\T^3}(\nabla \cdot u^n)\rho^{n+1}|u^{n+1}|^2 \,dx - \int_{\T^3 \times \R^3} \rho^{n+1}(u^n-v)f^{n+1}\cdot u^{n+1}\,dxdv \\
&\qquad =: \mathcal I_1 + \mathcal I_2,
\end{align*}
where $\mathcal I_1$ can be easily estimated as
\begin{align*}
\mathcal I_1 \leq \frac 12 \| \nabla u^n\|_{L^\infty} \int_{\T^3} \rho^{n+1}|u^{n+1}|^2 \,dx.
\end{align*}
For $\mathcal I_2$, we obtain
\begin{align*}
\mathcal I_2 &\leq \int_{\T^3 \times \R^3}\rho^{n+1}|u^n||u^{n+1}|f^{n+1}\,dxdv + \int_{\T^3 \times \R^3} \rho^{n+1}|v||u^{n+1}|f^{n+1}\,dxdv \\
&\leq \left(\int_{\T^3} \rho^{n+1}|u^{n+1}|^2 \,dx \right)^{1/2}\left(\int_{\T^3} \rho^{n+1}|u^n|^2 \lt(\int_{\R^3} f^{n+1}\,dv \rt)^2 dx\right)^{1/2} \\
&\quad + \left(\int_{\T^3} \rho^{n+1}|u^{n+1}|^2 \,dx \right)^{1/2} \left(\rho^{n+1}\lt( \int_{\R^3} |v|f^{n+1}\,dv\rt)^2 dx \right)^{1/2} \\
&\leq C \left(\int_{\T^3} \rho^{n+1}|u^{n+1}|^2 \,dx \right)^{1/2} (\|u^n\|_{L^\infty}+1)\|\rho^{n+1}\|_{L^\infty}^{1/2} \|f^{n+1}\|_q \\
&\leq \frac 12 \int_{\T^3} \rho^{n+1}|u^{n+1}|^2 \,dx + C (\|u^n\|_{L^\infty}+1)^2 \|\rho^{n+1}\|_{L^\infty} \|f^{n+1}\|_q^2.
\end{align*}
Then, we have
\begin{align*}
&\frac{d}{dt} \int_{\T^3} \rho^{n+1}|u^{n+1}|^2 \,dx + \int_{\T^3} | \nabla u^{n+1}|^2 \,dx \\
&\quad \leq (\| \nabla u^n \|_{L^\infty} +1) \int_{\T^3} \rho^{n+1} |u^{n+1}|^2 \,dx + C (\|u^n\|_{L^\infty}^2 +1) \| \rho^{n+1}\|_{L^\infty} \|f^{n+1}\|^2_q \\
&\quad \leq C( \| \nabla u^n\|_{H^2} + 1) \int_{\T^3} \rho^{n+1}|u^{n+1}|^2 \,dx + C \| \rho_0\|_{L^\infty}( \varepsilon^{2\alpha}+1) \varepsilon^{2\beta}.
\end{align*}
We now use Gr\"onwall's lemma and Lemma \ref{lem_rho} to obtain
\begin{align*}
& \int_{\T^3} \rho^{n+1}|u^{n+1}|^2 \,dx +\int_0^t \int_{\T^3} | \nabla u^{n+1}|^2 \,dxds \\
&\quad \leq \left( \int_{\T^3} \rho_0 |u_0|^2 \,dx \right) e^{\int_0^t C(\| \nabla u^n\|_{H^2}+1)\,ds} + C \|\rho_0\|_{L^\infty}(\varepsilon^{2\alpha}+1)\varepsilon^{2\beta} \int_0^t e^{\int_s^t C( \| \nabla u^n\|_{H^2} +1)\,d\tau}\,ds
\end{align*}
The exponential terms are estimated as follows:
Since
\begin{equation*}
\frac{1}{T} \left( \int_0^T \|u^n\|_{H^3}\,dt \right)^2 \leq \int_0^T \|u^n\|^2_{H^3} \,dt < \varepsilon^{2\alpha},
\end{equation*}
we have
\begin{equation*}
e^{C\int_0^t (1+\| \nabla u^n\|_{H^2})\,ds} \leq e^{C(T + \sqrt{T}\varepsilon^\alpha)} < e^{C(T+\sqrt{T})}.
\end{equation*}
Therefore, we have
\begin{align*}
\|u\|_{L^\infty(0,T;L^2)} + \| \nabla u\|_{L^2(0,T;L^2)} \leq C(\varepsilon+\varepsilon^{\alpha+\beta }+\varepsilon^{ \beta}) < \frac{\varepsilon^{ \alpha}}{10},
\end{align*}
where we used the smallness of $\varepsilon$. \newline

\noindent $\bullet$ (Step B): We take an inner product of both sides of $\eqref{app_main}_3$ with $\pa_t u^{n+1}$ and integrate it over $\T^3$  to find that
\begin{align*}
& \int_{\T^3}\rho^{n+1} |\pa_t u^{n+1}|^2\,dx + \frac{1}{2} \frac{d}{dt} \int_{\T^3}| \nabla u^{n+1}|^2\,dx \\
&\quad = - \int_{\T^3}\rho^{n+1}(u^n \cdot \nabla u^{n+1})\cdot \pa_t u^{n+1}\, dx - \int_{\T^3}\rho^{n+1} \pa_t u^{n+1}\lt(\int_{\R^3}(u^n-v)f^{n+1}\,dv\rt)dx \\
&\quad \leq \int_{\T^3}\rho^{n+1}|\pa_t u^{n+1}| |u^n \cdot \nabla u^{n+1}|  \,dx + \int_{\T^3} \rho^{n+1}| \pa_t u^{n+1}| \left|\int_{\R^3}(u^n-v)f^{n+1}\,dv \right|  dx \\
&\quad \leq \int_{\T^3}\rho^{n+1}\left(\frac{|\pa_t u^{n+1}|^2}{4} + |u^n|^2 |\nabla u^{n+1}|^2 \right) dx \\
&\qquad + \int_{\T^3} \rho^{n+1} \left(\frac{|\pa_t u^{n+1}|^2}{4} + \left|\int_{\R^3}(u^n-v)f^{n+1}\,dv \right|^2  \right)  dx\\
&\quad = \frac{1}{2} \int_{\T^3}\rho^{n+1} |\pa_t u^{n+1}|^2 \,dx + \int_{\T^3} \rho^{n+1}|u^n|^2 |\nabla u^{n+1}|^2 \,dx + \int_{\T^3} \rho^{n+1}  \left|\int_{\R^3}(u^n-v)f^{n+1}\,dv \right|^2 dx.
\end{align*}
So, we have
\begin{align} \label{C-4}
\begin{aligned}
& \| \sqrt{\rho^{n+1} }\pa_t u^{n+1}\|_{L^2}^2   +  \frac{d}{dt} \| \nabla u^{n+1}\|_{L^2}^2   \\
&\qquad \quad \quad \leq 2\int_{\T^3} \rho^{n+1}|u^n|^2 |\nabla u^{n+1}|^2 \,dx + 2\int_{\T^3} \rho^{n+1}  \left|\int_{\R^3}(u^n-v)f^{n+1}\,dv \right|^2 dx.
\end{aligned}
\end{align}
We note that the linearized momentum equations $\eqref{app_main}_3$ and $\eqref{app_main}_4$ can be written as the stationary Stokes equations
\begin{equation} \label{C-5-0}
-\Delta u^{n+1} + \nabla p^{n+1} = -\rho^{n+1} \left(\pa_t u^{n+1} -  u^n \cdot \nabla u^{n+1} -  \int_{\R^3}(u^n-v)f^{n+1}\,dv\right), \quad \nabla \cdot u = 0.
\end{equation}
Then, we get
\begin{align} \label{C-5}
\begin{aligned}
&\| \nabla^2 u^{n+1}\|_{L^2}^2 + \| \nabla p^{n+1}\|_{L^2}^2 \\
&\quad \leq C \lt\| -\rho^{n+1} \pa_t u^{n+1} - \rho^{n+1} u^n \cdot \nabla u^{n+1} - \rho^{n+1} \int_{\R^3} (u^n-v)f^{n+1} \,dv \rt\|_{L^2}^2 \\
&\quad \leq C \| \rho_0\|_{H^2} \left( \|\sqrt{\rho^{n+1}}\pa_t u^{n+1}\|_{L^2}^2 + \| \sqrt{\rho^{n+1}} u^n \cdot \nabla u^{n+1}\|_{L^2}^2 + \left\| \sqrt{\rho^{n+1}}\int_{\R^3}(u^n-v)f^{n+1}\,dv\right\|_{L^2}^2 \right).
\end{aligned}
\end{align}

It follows from \eqref{C-4} and \eqref{C-5} that
\begin{align} \label{C-6}
\begin{aligned}
&\| \sqrt{\rho^{n+1} }\pa_t u^{n+1}\|_{L^2}^2   + \| \nabla^2 u^{n+1}\|_{L^2}^2 + \| \nabla p^{n+1}\|_{L^2}^2 +  \frac{d}{dt} \| \nabla u^{n+1}\|_{L^2}^2 \\
&\qquad \leq C(1+ \| \rho_0\|_{H^2})^2 \left( \| u^n \cdot \nabla u^{n+1}\|_{L^2}^2 + \int_{\T^3} \left| \int_{\R^3} (u^n-v)f^{n+1} \,dv \right|^2 dx \right) \\
&\qquad \leq C(1+ \| \rho_0\|_{H^2})^2 ( \varepsilon^{2\alpha} \| \nabla u^{n+1}\|_{L^2}^2 + \varepsilon^{2\beta}).
\end{aligned}
\end{align}
Here, we used that
\begin{align*}
 \lt\| \int_{\R^3}(u^n-v)f^{n+1}\,dv \rt\|_{L^\infty} &\leq \lt\| \int_{\R^3} u^n f^{n+1}\,dv\rt\|_{L^\infty} + \lt\|\int_{\R^3} vf^{n+1}\,dv \rt\|_{L^\infty} \\
&\leq C\|u^n\|_{L^\infty}\|f^{n+1}\|_q + C\|f^{n+1}\|_q \\
&<C(\|u^n\|_{L^\infty}+1) \varepsilon^{2\beta} \\
&< C(\varepsilon^\alpha +1)\varepsilon^{2\beta}.
\end{align*}
for the last inequality. We now use Gr\"onwall's lemma to \eqref{C-6} to get
\begin{align*}
&\int_0^t \left( \| \sqrt{\rho^{n+1} }\pa_s u^{n+1}\|_{L^2}^2   + \| \nabla^2 u^{n+1}\|_{L^2}^2 + \| \nabla p^{n+1}\|_{L^2}^2 \right)ds + \| \nabla u^{n+1}\|_{L^2}^2 \\
&\quad \leq  C(1+\|\rho_0\|_{H^2}^2)(\varepsilon^{4\alpha} + \varepsilon^{2\beta}).
\end{align*}
Finally, we take supremum over $0 \leq t \leq T$ to obtain the desired result. \newline

\noindent $\bullet$ (Step C):  Note that
\begin{equation} \label{C-7}
\frac{d}{dt} \int_{\T^3} \rho^{n+1}|\pa_t u^{n+1}|^2 \,dx = \int_{\T^3}\pa_t \rho^{n+1}|\pa_t u^{n+1}|^2 \,dx + 2\int_{\T^3}\rho^{n+1}\pa_t u^{n+1} \cdot \pa_t^2 u^{n+1} \,dx.
\end{equation}
The first term can be estimated as follows.
$$\begin{aligned}
&\int_{\T^3}\pa_t\rho^{n+1}|\pa_t  u^{n+1}|^2 \,dx \cr
&\quad = -\int_{\T^3}(u^n \cdot \nabla \rho^{n+1})|\pa_t u^{n+1}|^2 \,dx \\
&\quad = \int_{\T^3}(\nabla \cdot u^n)\rho^{n+1}|\pa_t u^{n+1}|^2 \,dx + 2\int_{\T^3}\rho^{n+1}\pa_t u^{n+1} \cdot(u^n \cdot \nabla \pa_t u^{n+1})\,dx \\
&\quad \leq \| \nabla u^n\|_{L^\infty}  \int_{\T^3}\rho^{n+1} |\pa_t u^{n+1}|^2 \,dx +  2\int_{\T^3}\rho^{n+1}\pa_t u^{n+1} \cdot(u^n \cdot \nabla \pa_t u^{n+1}) \,dx.
\end{aligned}$$
We now give the estimates for the second term in \eqref{C-7}.
In view of $\eqref{app_main}_2$,  differentiating $\eqref{app_main}_3$ with respect to $t$ yields
\begin{align} \label{C-8}
\begin{aligned}
&\rho^{n+1}\pa_t^2 u^{n+1} \\
&\quad= (u^n \cdot \nabla \rho^{n+1})\lt(\pa_t u^{n+1} + u^n \cdot \nabla u^{n+1} + \int_{\R^3}(u^n-v)f^{n+1}\,dv\rt) \\
&\qquad - \rho^{n+1}\lt(\pa_t u^n \cdot \nabla u^{n+1} + u^n \cdot \nabla\pa_t u^{n+1}+ \int_{\R^3}f^{n+1} \pa_t u^n dv +\int_{\R^3} u^n \pa_t f^{n+1}\,dv\rt) \\
&\qquad - \nabla\pa_t p^{n+1} + \Delta\pa_t u^{n+1}.
\end{aligned}
\end{align}
Taking an inner product of both sides of $\eqref{C-8}$ with $\pa_t u^{n+1}$ and integrating it over $\T^3$  to obtain
\begin{align*}
&\int_{\T^3}\rho^{n+1}\pa_t u^{n+1} \cdot \pa_t^2 u^{n+1} \,dx \\
&\quad= \int_{\T^3}\pa_t u^{n+1} \cdot (u^n \cdot \nabla \rho^{n+1})\lt(\pa_t u^{n+1} + u^n \cdot \nabla u^{n+1} + \int_{\R^3}(u^n-v)f^{n+1}dv\rt)\,dx \\
&\qquad -\int_{\T^3}\pa_t u^{n+1} \cdot \rho^{n+1}\lt(\pa_t u^{n+1} \cdot \nabla u^{n+1} + u^n \cdot \nabla \pa_t u^{n+1} + \int_{\R^3}f^{n+1}\pa_t u^n \,dv + \int_{\R^3}u^n \pa_t f^{n+1}\,dv\rt) dx\\
&\qquad - \int_{\T^3}|\nabla \pa_t u^{n+1}|^2 \,dx \\
&=: \sum_{i=1}^7 \mathcal J_i.
\end{align*}
Here $\mathcal J_i, i=1,\cdots,7$ can be estimated as follows.
\begin{align*}
\mathcal J_1 &\leq \|\nabla \rho^{n+1}\|_{L^\infty} \|u^n\|_{L^\infty}\|\pa_t u^{n+1}\|_{L^2}^2 \leq C \|\pa_t u^{n+1}\|_{L^2}^2 \varepsilon^\alpha,\\
\mathcal J_2 &\leq \| \nabla \rho^{n+1}\|_{L^\infty}\|u^n\|_{L^\infty}^2 \|\pa_t u^{n+1}\|_{L^2}\| \nabla u^{n+1}\|_{L^2} \leq C\|\pa_t u^{n+1}\|_{L^2}^2\varepsilon^{3\alpha} \leq \|\pa_t u^{n+1}\|_{L^2}^2 \varepsilon^{2\alpha} + C\varepsilon^{4\alpha},    \\
\mathcal J_3  &\leq \|\nabla \rho^{n+1}\|_{L^\infty}\|u^n\|_{L^\infty}\|\pa_t u^{n+1}\|_{L^2}\lt\| \int_{\R^3}(u^n-v)f^{n+1}\,dv \rt\|_{L^2}\\
 &\leq C\|\pa_t u^{n+1}\|_{L^2}(1+\varepsilon^\alpha)\varepsilon^{\alpha+2\beta} \leq \|\pa_t u^{n+1}\|_{L^2}^2 \varepsilon^\beta + C\varepsilon^{2\alpha+\beta},\\
 \mathcal J_4 &\leq \|\rho^{n+1}\|_{L^\infty} \|\pa_t u^{n+1}\|_{L^6} \|\pa_t u^{n}\|_{L^3}\|\nabla u^{n+1}\|_{L^2}\\
 &\leq C\|\nabla \pa_t u^{n+1}\|_{L^2}\|\pa_t u^n\|_{H^1} \|\nabla u^{n+1}\|_{L^2} \\
 &\leq C\|\pa_t u^n\|_{H^1}^2\varepsilon^\alpha + \frac 13 \|\pa_t \nabla u^{n+1}\|_{L^2}^2,\\
\mathcal J_5 &\leq \|\rho^{n+1}\|_{L^\infty} \|u^n\|_{L^\infty} \|\pa_t u^{n+1}\|_{L^2}\|\nabla \pa_t u^{n+1}\|_{L^2} \\
 &\leq C\| \pa_t u^{n+1}\|_{L^2}^2\varepsilon^\alpha + \frac 13 \| \nabla \pa_t u^{n+1}\|_{L^2}^2,\\
\mathcal J_6  &\leq C\|\rho_0\|_{H^3} \|\pa_t u^{n+1}\|_{L^6}\| \pa_t u^n\|_{L^2}\lt\|\int_{\R^3}f^{n+1}\,dv\rt\|_{L^3} \\
 &\leq C\|\nabla \pa_t u^{n+1}\|_{L^2}\|\pa_t u^n\|_{L^2}\|f^{n+1}\|_q \\
&\leq C\varepsilon^{2\alpha+2\beta}+\frac 13 \|\nabla \pa_t u^{n+1}\|_{L^2}^2, \\
\mathcal J_7 &= \int_{\T^3}\pa_t u^{n+1} \cdot \rho^{n+1}\lt(\int_{\R^3}u^n\lt(v \cdot \nabla f^{n+1} + \nabla_v \cdot(\rho^{n+1}(u^n-v)f^{n+1})\rt)dv       \rt)   dx \\
&= \int_{\T^3 \times \R^3}\rho^{n+1}(\pa_t u^{n+1} \cdot u^n)(v\cdot \nabla f^{n+1})\,dxdv \\
&\quad+ \int_{\T^3 \times \R^3} (\rho^{n+1})^2 (\pa_t u^{n+1} \cdot u^n)\big((u^n-v)\cdot \nabla_v f ^{n+1}-3f^{n+1})\,dxdv\\
&\leq C\|\rho_0\|_{H^3}\|\pa_t u^{n+1}\|_{L^2}\|u^n\|_{L^2}\lt\| \int_{\R^3}v\cdot \nabla f^{n+1}\,dv\rt\|_{L^\infty} \\
&\quad + C\| \rho_0\|_{H^3}^2\|\pa_t u^{n+1}\|_{L^2} \|u^n\|_{L^2}\left(\left\| \int_{\R^3}(u^n-v)\cdot \nabla_v f^{n+1} \, dv\right\|_{L^\infty} + 3\left\| \int_{\R^3}f^{n+1}\,dv\right\|_{L^\infty} \right) \\
&\leq C\|\pa_t u^{n+1}\|_{L^2}\|u^n\|_{L^2}\|\nabla f^{n+1}\|_q + C\|\pa_t u^{n+1}\|_{L^2}\|u^n\|_{L^2}\lt((\|u^n\|_{L^\infty}+1)\| \nabla_v f^{n+1}\|_q+\|f^{n+1}\|_q \rt)\\
&\leq C\| \pa_t u^{n+1}\|_{L^2}\varepsilon^{\alpha+\beta} + C\|\pa_t u^{n+1}\|_{L^2}(\varepsilon^{2\alpha+\beta} + \varepsilon^{\alpha+\beta}) \\
&\leq C(\varepsilon^{2\beta}+\varepsilon^{4\alpha}) + (\varepsilon^{2\alpha}+\varepsilon^{2\beta})\|\pa_t u^{n+1}\|_{L^2}^2.
\end{align*}
We sum up the estimates above and use \eqref{C-7} to obtain
\begin{align*}
&\frac{d}{dt}\int_{\T^3}\rho^{n+1}|\pa_t u^{n+1}|^2\,dx + \int_{\T^3}|\nabla \pa_t u^{n+1}|^2\,dx\\
&\quad \leq \| \nabla u^n\|_{L^\infty}\int_{\T^3}\rho^{n+1} |\pa_t u^{n+1}|^2\,dx + C \| \pa_t u^{n+1}\|_{L^2}^2 \varepsilon^{\alpha} + C\|\pa_t u^n\|_{H^1}^2 \varepsilon^{\alpha} + C(\varepsilon^{2\beta}+\varepsilon^{4\alpha}).
\end{align*}
Using Gr\"onwall's lemma, we get
\begin{align*}
&\int_{\T^3}\rho^{n+1}|\pa_t u^{n+1}|^2\,dx + \int_0^t \| \nabla \pa_s u^{n+1}\|_{L^2}^2\,ds \\
&\quad\leq  \int_{\T^3}\rho_0^{n+1}| \pa_t u^{n+1}|_{t \to 0^+}^2 \,dx\,\exp\lt(\int_0^T \| \nabla u^n\|_{L^\infty}\,ds\rt)\\
&\qquad +C \int_0^t (\|\pa_s u^{n+1}\|_{L^2}^2 \varepsilon^\alpha + \| \pa_s u^n\|_{H^1}^2 \varepsilon^\alpha + \varepsilon^{2\alpha})\exp\lt(\int_s^t \|\nabla u^n\|_{L^\infty} \,d\tau\rt)ds \\
&\quad \leq C\int_{\T^3}\rho_0^{n+1}| \pa_t u^{n+1}|_{t \to 0^+}^2 \,dx \exp(\| \nabla u^n\|_{L^2(0,T;L^\infty)}) \\
&\qquad+ C(\| \pa_t u^{n+1}\|^2_{L^2(0,T;L^2)}\varepsilon^\alpha + \| \pa_t u^n\|^2_{L^2(0,T;H^1)}\varepsilon^\alpha + \varepsilon^{2\beta}+\varepsilon^{4\alpha})\exp(\| \nabla u^n\|_{L^2(0,T;L^\infty)})\\
&\leq C\int_{\T^3}\rho_0^{n+1}| \pa_t u^{n+1}|_{t \to 0^+}^2 \,dx + C(\varepsilon^{3\alpha} + \varepsilon^{2\beta}),
\end{align*}
where we used the smallness of $\varepsilon$ for the last inequality. We can also derive the following estimates similarly:
\begin{align*}
\int_{\T^3}\rho^{n+1}|\pa_t u^{n+1}|^2 \,dx &= \int_{\T^3} \left(\rho^{n+1} \int_{\R^3} (v-u^n)f^{n+1}\,dv-\rho^{n+1} u^n \cdot \nabla u^{n+1} + \Delta u^{n+1} \right) \cdot \pa_t u^{n+1} \,dx \\
&\leq C(\varepsilon^{2\beta} + \varepsilon^{4\alpha}) + C\int_{\T^3} | \Delta u^{n+1}|^2 \,dx,
\end{align*}
which readily gives
\begin{align*}
\int_{\T^3} \rho_0^{n+1} |\pa_t u^{n+1}|_{t \to 0^+}^2 \,dx &\leq C(\varepsilon^{2\beta} + \varepsilon^{4\alpha}) + C\int_{\T^3} | \Delta u_0|^2 \,dx \\
&<C(\varepsilon^{2\beta} + \varepsilon^{4\alpha} + \varepsilon^{2}) <C(\varepsilon^{2\beta} + \varepsilon^{4\alpha}).
\end{align*}
Finally, we obtain that
\begin{align*}
\|\pa_t u^{n+1}\|_{\mathcal C(0,T];L^2)} + \| \nabla \pa_t u^{n+1}\|_{L^2(0,T;L^2)} < C(\varepsilon^{\beta}+\varepsilon^{\frac 32 \alpha}) < \varepsilon^{\alpha^*},
\end{align*}
where $\alpha < \alpha^* < \min \{\beta, (3\alpha)/2 \}$. \newline

\noindent $\bullet$ (Step D): We get from \eqref{C-7} that
\begin{align*}
&\| \nabla^2 u^{n+1}\|_{\mathcal C([0,T];L^2)}^2 + \| \nabla p^{n+1}\|_{\mathcal C([0,T];L^2)}^2 \\
&\quad \leq C \| \rho_0\|_{H^2} \bigg( \|\sqrt{\rho^{n+1}}\pa_t u^{n+1}\|_{\mathcal C([0,T];L^2)}^2  + \| \sqrt{\rho^{n+1}} u^n \cdot \nabla u^{n+1}\|_{\mathcal C([0,T];L^2)}^2  \\
&\qquad + \left\| \sqrt{\rho^{n+1}}\int_{\R^3}(u^n-v)f^{n+1}\,dv\right\|_{\mathcal C([0,T];L^2)}^2  \bigg)\\
&\quad \leq C\Big(\| \pa_t u^{n+1}\|_{\mathcal C(0,T;L^2)}^2 + \|u^n\|_{\mathcal C(0,T;L^\infty)}^2 \| \nabla u^{n+1}\|_{\mathcal C(0,T;L^2)}^2 + (\| u^n\|_{\mathcal C(0,T;L^\infty)}^2+1) \|f\|_{L^\infty(0,T;L^\infty_q)}^2\Big)\\
&\quad \leq C(\varepsilon^{2\alpha^*  } + \varepsilon^{4\alpha  } + \varepsilon^{2\alpha+2\beta  }  + \varepsilon^{ 2\beta }),
\end{align*}
which readily gives that
\begin{equation} \label{C-9}
\| \nabla^2 u^{n+1}\|_{\mathcal C([0,T];L^2)} + \| \nabla p^{n+1}\|_{\mathcal C([0,T];L^2)} < \frac{\varepsilon^{ \alpha}}{20}.
\end{equation}
We now give the estimates of \eqref{C-5-0} for the higher regularity.
\begin{align} \label{C-10}
\begin{aligned}
&\| \nabla^3 u^{n+1}\|_{L^2}^2 + \| \nabla^2 p^{n+1}\|_{L^2}^2 \\
&\quad \leq C\left( \| \nabla(\rho^{n+1}\pa_t u^{n+1})\|_{L^2}^2 +\| \nabla(\rho^{n+1} u^n \cdot \nabla u^{n+1})\|_{L^2}^2 +\lt\| \nabla \lt(\rho^{n+1} \int_{\R^3}(v-u^n)f^{n+1} \,dv \rt)\rt\|_{L^2}^2 \right)\\
&\quad =: \sum_{i=1}^3 \mathcal K_i.
\end{aligned}
\end{align}
Using the previous steps and \eqref{C-9}, we get
\begin{align*}
&\mathcal K_1 \leq C(\| \pa_t u^{n+1}\|_{L^2}^2 + \| \nabla \pa_t u^{n+1}\|_{L^2}^2) \leq C(\varepsilon^{2\beta} +  \| \nabla \pa_t u^{n+1}\|_{L^2}^2),\\
&\mathcal K_2 \leq C(\| u^n\|_{L^\infty}^2 \| \nabla u\|_{L^2}^2 + \| \nabla u^n\|_{L^\infty}^2 \| \nabla u^{n+1}\|_{L^2}^2 + \|u^n\|_{L^\infty}^2 \| \nabla^2 u^{n+1}\|_{L^2}^2) \\
&\quad \leq C(\| \nabla u^n\|_{L^\infty}^2 \varepsilon^{2\alpha} + \varepsilon^{4\alpha}),\\
&\mathcal K_3 \leq C \lt\| \nabla\lt(\rho^{n+1} \int_{\R^3}vf^{n+1}\,dv\rt)\rt\|_{L^\infty}^2 + C\lt\|\nabla \lt( \rho^{n+1} \int_{\R^3}u^n f^{n+1} \,dv\rt)\rt\|_{L^\infty}^2  \\
&\quad \leq C\bigg(\lt\| \int_{\R^3} vf^{n+1} \,dv\rt\|_{L^\infty}^2 + \lt\| \int_{\R^3}v\nabla f^{n+1}\,dv\rt\|_{L^\infty}^2 + \lt\| \int_{\R^3} u^n f^{n+1}\,dv    \rt\|_{L^\infty}^2 \\
&\qquad + \lt\| \int_{\R^3}f^{n+1} \nabla u^n  \,dv    \rt\|_{L^\infty}^2 + \lt\| \int_{\R^3}u^n \nabla f^{n+1}\,dv \rt\|_{L^\infty}^2\bigg) \\
&\quad \leq C(\varepsilon^{2\beta} + \varepsilon^{2\alpha+ 2\beta} + \varepsilon^{2\beta}\| \nabla u^n\|_{L^\infty}^2).
\end{align*}
Combining this with \eqref{C-10} and using \eqref{C-9} again, we have
\begin{align*}
&\| \nabla^3 u^{n+1}\|_{L^2(0,T;L^2)}^2 + \| \nabla^2 p^{n+1}\|_{L^2(0,T;L^2)}^2 \\
&\qquad \leq C\lt( \| \nabla \pa_t u^{n+1}\|_{L^2(0,T;L^2)}^2 + ( \varepsilon^{2\alpha}+\varepsilon^{2\beta} ) \|  u^n\|_{L^2(0,T;H^3)}^2+ \varepsilon^{2\beta} + \varepsilon^{4\alpha} + \varepsilon^{2\alpha+2\beta}\rt)\\
&\qquad \leq C(\varepsilon^{ 4\alpha} + \varepsilon^{ 2\alpha+2\beta} + \varepsilon^{2\beta }+\varepsilon^{2\alpha^*  }),
\end{align*}
where $C > 0$ is independent of $n$. This gives
\begin{align*}
\| \nabla^3 u^{n+1}\|_{L^2(0,T;L^2)} + \| \nabla^2 p^{n+1}\|_{L^2(0,T;L^2)} \leq C(\varepsilon^{ 2\alpha} + \varepsilon^{ \alpha + \beta} + \varepsilon^{ \beta}+\varepsilon^{\alpha^*}) < \frac{\varepsilon^{ \alpha}}{20}.
\end{align*}

\section*{Acknowledgments}
YPC was supported by National Research Foundation of Korea(NRF) grant funded by the Korea government(MSIP) (No. 2017R1C1B2012918 and 2017R1A4A1014735) and POSCO Science Fellowship of POSCO TJ Park Foundation. Seok-Bae Yun is supported by Samsung Science and Technology Foundation under Project Number SSTF-BA1801-02.

%
%
%
%


\begin{thebibliography}{10}

\bibitem{BCHK} Bae, H.-O., Choi, Y.-P., Ha, S.-Y., and Kang, M.-J.: Time-asymptotic interaction of flocking particles and incompressible viscous fluid, Nonlinearity, {\bf25} (2012), 1155--1177.


\bibitem{BCHK14} Bae, H.-O., Choi, Y.-P., Ha, S.-Y., and Kang, M.-J.: Global existence of strong solution for the Cucker-Smale-Navier-Stokes system, J. Differential Equations, {\bf 257}, (2014), 2225--2255.

\bibitem{BY} Bang, J. and Yun, S.-B.: Stationary solutions for the ellipsoidal BGK model in a slab, J. Differential  Equations, {\bf261} (2016), 5803--5828.

\bibitem{BDM} Benjelloun, S., Desvillettes, L., and Moussa, A.: Existence theory for the kinetic-fluid coupling when small droplets are treated as part of the fluid, J. Hyperbolic Differ. Equ., {\bf11} (2014), 109--133.

\bibitem{V} Berthelin, F. and Vasseur, A.: From kinetic equations to multidimensional isentropic gas dynamics before shocks, SIAM J. Math. Anal., {\bf36} (2005), 1807--1835.

\bibitem{BGK}  Bhatnagar, P. L., Gross, E. P. and Krook, M.: A model for collision processes in gases. I. Small
amplitude process in charged and neutral one-component systems, Phys. Rev., {\bf 94} (1954), 511--525.

\bibitem{BDGM} Boudin, L., Desvillettes, L., Grandmont, C., and Moussa, A.: Global existence of solution for the coupled Vlasov and Naiver-Stokes equations, Differential and Integral Equations, {\bf 22} (2009), 1247--1271.

\bibitem{BDM} Boudin, L., Desvillettes, L. and Motte, R.: A modelling of compressible droplets in a fluid
Commun. Math. Sci., {\bf 1} (2003), 657--669.

\bibitem{CDM} Carrillo, J. A., Duan, R., and Moussa, A.: Global classical solutions close to the equilibrium to the Vlasov-Fokker-Planck-Euler system, Kinet. Relat. Models, {\bf 4} (2011), 227--258.

\bibitem{CCK} Carrillo, J. A., Choi, Y.-P., and Karper, T. K.: On the analysis of a coupled kinetic-fluid model with local alignment forces, Ann. Inst. Henri Poincar\' e Anal. Non Lin\'eaire, {\bf 33} (2016), 273--307.


\bibitem{CKL} Chae, M., Kang, K., and Lee, J.: Global classical solutions for a compressible
uid-particle interaction model, J. Hyperbolic Differ. Equ., {\bf 10} (2013), 537--562.

\bibitem{Choi16} Choi, Y.-P.: Large-time behavior of the Vlasov/compressible Navier-Stokes equations: J. Math. Phys., {\bf 57} (2016), 071501.

\bibitem{Choi17} Choi, Y.-P.: Finite-time blow-up phenomena of Vlasov/Navier-Stokes equations and related systems, J. Math. Pures Appl., {\bf 108} (2017), 991--1021.

\bibitem{CK} Choi, Y.-P. and Kwon, B.: Global well-posedness and large-time behavior for the inhomogeneous Vlasov-Navier-Stokes equations, Nonlinearity, {\bf 28} (2015), 3309--3336.

\bibitem{CL} Choi, Y.-P. and Lee, J.: Global existence of weak and strong solutions to Cucker-Smale-Navier-Stokes
equations in $\R^2$, Nonlinear Analysis: Real World Applications, {\bf27} (2016), 158--182.

\bibitem{CY} Choi, Y.-P. and Yun, S.-B.: Global existence of weak solutions for Navier-Stokes-BGK system, preprint, arXiv:1801.08283

\bibitem{CP} Coron, F. and Perthame, B. : Numerical passage from kinetic to fluid equations, SIAM J. Numer. Anal. {\bf28} (1991),  26--42.

\bibitem{DP} Dimarco, G. and Pareschi, L.: Numerical methods for kinetic equations. Acta Numer., {\bf23} (2014), 369--520.

\bibitem{DMOS} Dolbeault, J., Markowich, P., Oelz, D., Schmeiser, C.: Non linear diffusions as limit of kinetic equations with relaxation collision kernels. Arch. Ration. Mech, Anal. {\bf 186} (2007), 133--158.

\bibitem{FJ} Filbet, F. and Jin, S.: A class of asymptotic-preserving schemes for kinetic equations and related problems with stiff sources. J. Comput. Phys., {\bf229} (2010), 7625--7648.

\bibitem{Ha} Hamdache, K.: Global existence and large time behaviour of solutions for the Vlasov-Stokes equations, Jpn. J. Ind. Appl. Math., {\bf 15} (1998), 51--74.

\bibitem{Ka} Kato, T.: Linear evolution equations of ‘hyperbolic’ type, II, J. Math. Soc. Japan, {\bf 25} (1973), 648--666.

\bibitem{KP} Klingenberg, C. and Pirner, M.: Existence, uniqueness and positivity of solutions for BGK models for mixtures, J. Differential Equations, {\bf 264} (2018), 702--727.

\bibitem{LT} Lions, P. L. and Toscani, G.: Diffusive limit for finite velocity Boltzmann kinetic models, Rev. Mat. Iberoamericana, {\bf 13} (1997), 473--513.

\bibitem{M} Mathiaud, J.: Local smooth solutions of a thin spray model with collisions, Math. Mod. Meth. Appl. Sci.,
{\bf20} (2010), 191--221.

\bibitem{Mellet} Mellet, A.: Fractional diffusion limit for collisional kinetic equations: a moments method, Indiana Univ. Math. J., {\bf 59} (2010), 1333--1360.

\bibitem{MMM} Mellet, A., Mischler, S., and Mouhot, C.: Fractional diffusion limit for collisional kinetic equations, Arch. Ration. Mech. Anal., {\bf 199} (2011), 493--525.

\bibitem{MV} Mellet, A. and Vasseur, A.: Global weak solutions for a Vlasov-Fokker-Planck/Navier-Stokes system of equations, Math. Models Methods Appl. Sci., {\bf 17} (2007), 1039--1063.

\bibitem{M} Mieussens, L.: Discrete velocity model and implicit scheme for the BGK equation of rarefied gas dynamics, Math. Models Methods Appl. Sci., {\bf10} (2000), 1121--1149.

\bibitem{PY} Park, S. and Yun, S.-B.: Cauchy problem for ellipsoidal BGK model for polyatomic particles,  J. Differential Equations, {\bf 266} (2019), 7678--7708.

\bibitem{Pe} Perthame, B.: Global existence to the BGK model of Boltzmann equation, J. Differential Equations, {\bf 82} (1989), 191--205.

\bibitem{PP93} Perthame, B. and Pulvirenti, M.: Weighted $L^\infty$ bounds and uniqueness for the Boltzmann BGK model, Arch. Rational Mech. Anal., {\bf 125} (1993), 289--295.

\bibitem{PP07} Pieraccini, S. and Puppo, G.: Implicit-explicit schemes for BGK kinetic equations, J. Sci. Comput., {\bf32} (2007), 1--28.

\bibitem{RSY} Russo, G., Santagati, P. and Yun, S.-B.: Convergence of a semi-Lagrangian scheme for the BGK model of the Boltzmann equation, SIAM J. Numer. Anal., {\bf50} (2012), 1111--1135.

\bibitem{SR} Saint-Raymond, L.: Discrete time Navier-Stokes limit for the BGK Boltzmann equation, Comm. Partial Differential Equations, {\bf 27} (2002), 149--184.

\bibitem{SR1} Saint-Raymond, L.: Du mod\`{e}le BGK de l'\`{e}quation de Boltzmann aux \`{e}quations d'Euler des fluides incompressibles. (French) [From the BGK Boltzmann model to the Euler equations of incompressible fluids], Bull. Sci. Math., {\bf 126} (2002), 493--506.


\bibitem{SR2} Saint-Raymond, L.: From the BGK model to the Navier-Stokes equations, Ann. Sci. I'\'Ecole Norm. Sup., {\bf 36} (2003), 271--317.

\bibitem{WY} Wang, D. and Yu, C.: Global weak solutions to the inhomogeneous Navier-Stokes-Vlasov equations, J. Differential Equations, {\bf 259} (2014), 3976--4008.

\bibitem{Ukai-BGK} Ukai, S.: Stationary solutions of the BGK model equation on a finite interval with large boundary data, Transport theory Statist. Phys., {\bf 21} (1992) no.4-6.

\bibitem{XH} Xu, K. and Huang, J.-C.: A unified gas-kinetic scheme for continuum and rarefied flows. J. Comput. Phys., {\bf229} (2010), 7747--7764.

\bibitem{YY}Yao, L. and Yu, C.: Existence of global weak solutions for the Navier-Stokes-Vlasov-Boltzmann equations, J. Differential Equations, {\bf265} (2018), 5575--5603.

\bibitem{Yun1} Yun, S.-B.: Cauchy problem for the Boltzmann-BGK model near a global Maxwellian, J. Math. Phys., {\bf 51} (2010), 123514.

\bibitem{Yun2} Yun, S.-B.: Classical solutions for the ellipsoidal BGK model with fixed collision frequency, J. Differential Equations,  {\bf 259} (2015), 6009--6037.

\bibitem{Yun21} Yun, S.-B.: Ellipsoidal BGK model for polyatomic molecules near Maxwellians: A dichotomy in the dissipation estimate, J. Differential Equations, {\bf 266} (2019), 5566--5614.

\bibitem{Yun3} Yun, S.-B.: Ellipsoidal BGK model near a global Maxwellian, SIAM J. Math. Anal., {\bf 47} (2015), 2324--2354.

\bibitem{Zhang} Zhang, X.: On the Cauchy problem of the Vlasov-Posson-BGK system: global existence of weak solutions, J. Stat. Phys., {\bf141} (2010), 566--588.

\bibitem{ZH} Zhang, X. and Hu, S.: $L^p$ solutions to the Cauchy problem of the BGK equation, J. Math. Phys., {\bf 48} (2007), 113304.

\end{thebibliography}
\end{document}